\documentclass{article}

\usepackage[a4paper,
 total={140mm,237mm},
 left=35mm,
 top=30mm,
 ]{geometry}
\usepackage[T1]{fontenc}
\usepackage[utf8]{inputenc}
\usepackage{mathpazo}
\usepackage[hidelinks]{hyperref}
\usepackage{doi}
\usepackage{relsize}

\usepackage{amssymb}
\usepackage{amsmath}
\usepackage[dvipsnames]{xcolor}
\usepackage[english]{babel}
\usepackage{amsthm}
\usepackage{enumitem}
\usepackage{tikz}
\usetikzlibrary{arrows}
\usetikzlibrary{calc}
\usepackage{bm}
\usepackage{csquotes}

\usepackage[shortcuts]{extdash}
\usepackage{authblk}
\usepackage{rotating}
\usepackage{caption}
\usepackage{tabularx}
\usepackage{multirow}

\usetikzlibrary {shapes.misc}

\usepackage{extarrows}

\usetikzlibrary{arrows.meta}
\usepackage[shortcuts]{extdash}

\usepackage[format=plain,
            labelfont=it,
            textfont=it]{caption}
\usepgflibrary{shapes.geometric}
\usepackage{float}
\allowdisplaybreaks

\usepackage[maxnames=20,backend=biber, url=false, doi=false, eprint=false]{biblatex}

\newtheorem{theorem}{Theorem}[section]

\newtheorem{corollary}[theorem]{Corollary}

\newtheorem{lemma}[theorem]{Lemma}
\newtheorem{proposition}[theorem]{Proposition}

\theoremstyle{definition}
\newtheorem{example}[theorem]{Example}
\newtheorem{definition}[theorem]{Definition}
\newtheorem{question}[theorem]{Question}

\usepackage{rotating}
\usepackage{tikz}
\usepackage{caption}
\usepackage{tabularx}
\usepackage{multirow}
\usepackage{bm}

\usetikzlibrary {shapes.misc}
\usetikzlibrary{arrows}

\usepackage{extarrows}

\usetikzlibrary{arrows.meta}
\usepackage[shortcuts]{extdash}

\newcommand{\ignore}[1]{}

\newtheorem{prop}[theorem]{Proposition}

\theoremstyle{remark}

\theoremstyle{definition}

\newcommand{\mic}[1]{\textcolor{black}{#1}}
\newcommand{\pau}[1]{\textcolor{black}{#1}}

\DeclareMathOperator{\ZZ}{Z}
\newcommand{\zar}{\tau_{\ZZ}}
\DeclareMathOperator{\ppww}{pw}
\newcommand{\pw}{\tau_{\ppww}}

\def\acts{\curvearrowright}

\DeclareMathOperator{\Pol}{Pol}
\DeclareMathOperator{\Aut}{Aut}
\DeclareMathOperator{\Autf}{Autf}
\DeclareMathOperator{\End}{End}
\DeclareMathOperator{\Emb}{Emb}
\DeclareMathOperator{\EEmb}{EEmb}
\DeclareMathOperator{\acl}{acl}
\newcommand{\sA}{\mathbb{A}}
\newcommand{\sB}{\mathbb{B}}
\newcommand{\sC}{\mathbb{C}}

\newcommand{\fullclone}{O}
\newcommand{\cP}{P}%{\mathcal{P}}
\newcommand{\cS}{S}
\newcommand{\cT}{T}
\newcommand{\cC}{C}%{\mathcal{C}}
\newcommand{\cD}{D}%{\mathcal{D}}
\newcommand{\cU}{U}
\newcommand{\classC}{\mathcal{C}}
\newcommand{\cG}{{G}}
\newcommand{\cK}{\mathcal{K}}

\usetikzlibrary{decorations.markings}
\tikzset{middlearrow/.style={
        decoration={markings,
            mark= at position 0.6 with {\arrow{#1}} ,
        },
        postaction={decorate}
    }
}

\newcommand{\cclosure}[1]{\langle{#1}\rangle}
\newcommand{\tcl}[1]{\overline{#1}}

\newcommand{\caut}{\classC_\cG^\mathrm{ol}}
\newcommand{\cgroup}
{\classC_\cG^\mathrm{cl}}
\newcommand{\csecondgroup}
{\classC_\cG^\mathrm{2}}
\newcommand{\csecondeemb}
{\classC_{\tcl{\cG}}^\mathrm{2}}
\newcommand{\ceemb}{\classC_{\tcl{\cG}}^{\mathrm{ol}}}
\newcommand{\cend}
{\classC_\cT^\mathrm{ol}}
\newcommand{\csemi}
{\classC_\cT^\mathrm{cl}}
\newcommand{\ctcl}
{\classC_\cT^\mathrm{cl}}
\newcommand{\csecondsemi}
{\classC_\cT^\mathrm{2}}
\newcommand{\cpol}
{\classC_\cC^\mathrm{ol}}
\newcommand{\cclone}
{\classC_\cC^\mathrm{cl}}
\newcommand{\csecondclone}
{\classC_\cC^\mathrm{2}}

\newcommand{\csatgroup}
{\classC_\cG^\mathrm{sat}}
\newcommand{\csateemb}
{\classC_{\tcl{\cG}}^\mathrm{sat}}
\newcommand{\csatsemi}
{\classC_\cT^\mathrm{sat}}
\newcommand{\csatclone}
{\classC_\cC^\mathrm{sat}}
\newcommand{\cclgroup}{\classC_{\tcl{\cG}}^{\mathrm{cl}}}

\newcommand{\cnagroup}
{\classC_\cG^\mathrm{na}}
\newcommand{\cnaeemb}
{\classC_{\tcl{\cG}}^\mathrm{na}}
\newcommand{\cnasemi}
{\classC_\cT^\mathrm{na}}
\newcommand{\cnaclone}
{\classC_\cC^\mathrm{na}}

\newcommand{\cpolgroup}
{\classC_\cG^\mathrm{Po}}
\newcommand{\cpoleemb}
{\classC_{\tcl{\cG}}^\mathrm{Po}}
\newcommand{\cpolsemi}
{\classC_\cT^\mathrm{Po}}
\newcommand{\cpolclone}
{\classC_\cC^\mathrm{Po}}

\newcommand{\tgcl}
{\tau_{c}}

\newcommand{\somega}
{S_\Omega}
\newcommand{\omom}
{\Omega^\Omega}

\begin{document}

\title{A guide to topological reconstruction on endomorphism monoids and polymorphism clones}

\author[1]{Paolo Marimon}
\author[2]{Michael Pinsker}
\affil[1]{Department of Computer Science, University of Oxford. Oxford, UK.}
\affil[2]{Institut f\"{u}r Diskrete Mathematik \& Geometrie, TU Wien. Vienna, Austria.}

\maketitle
\begin{abstract}
{Various spaces of symmetries of a structure are naturally endowed with both an algebraic and a topological structure. For example, the automorphism group of a structure is, {on top of being a group,} %, of course, a group, and 
 a topological group when equipped with the topology of pointwise convergence. In some cases, {the algebraic structure of such space alone is sufficiently rich} %such algebraic structure is rich enough that it actually 
 to determine its topology (under some requirements on the topology). For automorphism groups, the problem of when this happens has been actively pursued over the last 40 years. {With the exception of some early work of Lascar, the analogue of this problem for endomorphism monoids and polymorphism clones has only received attention in the past 15 years.}
 %\pau{The analogue of this problem for endomorphism monoids and polymorphism clones has only received attention in the past 15 years (with the exception of some early work of Lascar).}
 % With the exception of one set of older results due to Lascar, most results on endomorphism monoids and polymorphism clones are much younger and date from the past 15 years
% With the exception of some older results of Lascar, most work on endomorphism monoids and polymorphism clones is much younger and dates from the past 15 years.} %More recently, over the past 15 years, numerous results have been obtained also for endomorphism monoids and polymorphism clones. 
%In a sense, these spaces of symmetries have weaker algebraic structure, but encode more information regarding the original structures, giving rise to a rich repertoire of methods, often very different from those used for groups. 
In this guide, we survey the current state of affairs in this relatively young line of research. {We moreover use this opportunity to polish several existing results and to extend them beyond what was hitherto known.}} %In this process, we polished several existing results and connections, and we also prove some substantial results. In particular, we prove that automatic homeomorphicity lifts upwards from the automorphism group to the monoid of elementary embeddings of a countable saturated structure. This was an open problem that was at the center of several papers in the area, including~\cite{Reconstruction, pech-saturated, truss2021reconstructing}.}
\end{abstract}

\section{Introduction}

Given a structure $\sA$ with domain $\Omega$, we may associate with it various spaces of symmetries: its automorphism group $\Aut(\sA)$, its monoids of elementary self-embeddings $\EEmb(\sA)$, of self-embeddings $\Emb(\sA)$, and of endomorphisms $\End(\sA)$, and its polymorphism clone $\Pol(\sA)$ (the space of homomorphisms from some finite power of $\sA$ into itself, see Section~\ref{sec:spacesofsymm}).\\

Each space of symmetries of $\sA$ is naturally endowed with the topology of pointwise convergence $\pw$. For the spaces which are subsets of  $\Omega^\Omega$ {(which include all but the last space
%s 
above)}, this is just the subspace topology induced by the product topology on $\Omega^\Omega$, where $\Omega$ is endowed with the discrete topology; in the case of $\Pol(\sA)$ it is the topology induced in the sum space $\fullclone:=\bigcup_{n\in\mathbb{N}}\Omega^{\Omega^n}$ of the spaces $\Omega^{\Omega^n}$ which again bear the product topology. %Whilst we will focus on such spaces throughout this introduction, this topology may be defined analogously for polymorphism clones ({see BLAH}). 
It is folklore that the subgroups of the full symmetric group  $S_\Omega$ on $\Omega$ which are topologically closed in it  correspond to automorphism groups of first-order structures with domain $\Omega$. Analogously, closed submonoids of the full transformation monoid $\Omega^\Omega$ are precisely the  endomorphism monoids of first-order structures, % with domain $\Omega$, 
and closed subclones of the full clone $\fullclone:=\bigcup_{n\in\mathbb{N}}\Omega^{\Omega^n}$ the  polymorphism clones of first-order structures.\\% with domain $\Omega$.\\

When $\Omega$ is countable, as we will be assuming throughout %the rest of 
this paper,  closed submonoids of $\Omega^\Omega$ %(equipped with $\pw$) 
are Polish topological semigroups. That is, they are separable and completely metrizable (i.e. Polish) and the semigroup operation is continuous. Similarly, closed subgroups of $S_\Omega$ are Polish topological groups (that is, both group multiplication and inversion are continuous), and closed subclones of $\fullclone$ are Polish topological clones (that is, composition of functions is continuous).\\

The picture described above raises two natural %(though still somewhat vague) 
questions:
\begin{question}\label{q:reconstruction} What information can we recover about the original structure $\sA$ from a given space of symmetries (thought of as a topological group / monoid / clone)?
\end{question}
\begin{question}\label{q:algebratotopology} To what extent does the algebraic structure of a space of symmetries determine its  topological structure?  %In particular, are there other topologies in addition to $\tau_{pw}$ with respect to which the  automorphism group / endomorphism monoid / polymorphism clone of $\sA$ is a Polish topological group / monoid / clone?
\end{question}
Regarding Question~\ref{q:reconstruction}, we have an extremely satisfying answer when $\sA$ belongs to the class of $\omega$-categorical structures, which includes e.g.~$(\mathbb{N}; =)$, $(\mathbb{Q};<)$, the random graph, countable vector spaces over finite fields, and the countable atomless Boolean algebra $B_\infty$. %These are countable structures whose automorphism group has finitely many orbits on $\sA^n$ for each $n\in\mathbb{N}$. They include $(\mathbb{N}; =)$, $(\mathbb{Q};<)$, countable vector spaces over finite fields, and the countable atomless Boolean algebra $B_\infty$. When $\sA$ is $\omega$-categorical, its 
In that case, its spaces of symmetries %(thought as both algebraic and topological spaces) 
recover the original structure up to certain notions of bi-interpretation~\cite{AhlbrandtZiegler, BodJunker, Topo-Birk}. In particular, $\Aut(\sA)$ and $\EEmb(\sA)$ recover $\sA$ up to bi-interpretability, $\Pol(\sA)$ recovers $\sA$ up to the stronger notion of primitive positive bi-interpretability, and $\End(\sA)$ recovers it up to the intermediate notion of existential positive bi-interpretability (modulo some minor caveats, see Theorem~\ref{thm:biinterpretations}). Moreover, by a  result  of~\cite{Rubin}, restricting our attention to $\omega$-categorical structures with no algebraicity, their spaces of symmetries with the topologies of pointwise convergence recover the original structure even up to bi-definability \mic{(see also~\cite{RomanMichael})}.\\

%Again, spaces of symmetries of $\omega$-categorical structures have natural characterisations in terms of their actions: a countable structure $\sA$ is $\omega$-categorical if and only if $\Aut(\sA)$ is oligomorphic in the sense that $\Aut(\sA)\acts\Omega^n$ has finitely many orbits for each $n\in\mathbb{N}$. We call a closed transformation monoid $\cT\subseteq\Omega^\Omega$ oligomorphic if it contains an oligomorphic permutation group, and oligomorphic transformation monoids correspond to endomorphism monoids of countable $\omega$-categorical structures. Moreover, monoids of elementary embeddings of countable $\omega$-categorical structures correspond precisely to closed oligomorphic transformation monoids with a dense set of invertibles (note that this correspondence was missing in the absence of $\omega$-categoricity).\remmic{actually saturation is sufficient for this} Similarly, polymorphism clones of $\omega$-categorical structures correspond to oligomorphic clones, i.e., closed clones whose unary operations are an oligomorphic transformation monoid.\\

%The fact that (topological) spaces of symmetries of $\omega$-categorical structures recover the original structure up to some notion of bi-interpretability 
This fact makes Question~\ref{q:algebratotopology} even more inviting in %this 
the $\omega$-categorical 
context. There are several ways of making Question~\ref{q:algebratotopology} precise, which gives  rise to different notions of ``topological reconstruction''. Let us focus for a moment on a closed transformation monoid $\cS$ and on a class of topological semigroups $\classC$. We say that $\cS$ is \textbf{reconstructible} with respect to $\classC$ if whenever there is an algebraic isomorphism $\alpha:\cS\to\cT\in\classC$, there is also a topological isomorphism $\xi:\cS\to\cT$. If, instead, we ask the topological isomorphism between $\cS$ and $\cT$ is already given by $\alpha$, then we say that $\cS$ has \textbf{automatic homeomorphicity} (AH). Automatic homeomorphicity with respect to the class of Polish semigroups is known as the \textbf{Unique Polish Property} (UPP) since it is equivalent to  there being a unique Polish semigroup topology on $\cS$. Finally, one might ask that \mic{all semigroup homomorphisms} from $\cS$ to a member $\cT$ of $\classC$ are continuous. This property is known as \textbf{automatic continuity} (AC). All notions above can be defined analogously for topological groups and topological clones.\\

For  automorphism groups of $\omega$-categorical structures, these notions of reconstruction have been heavily studied in model theory, group theory, and topological dynamics. Reconstruction of $\Aut(\sA)$ with respect to the class $\caut$ of automorphism groups of $\omega$-categorical structures (also known as closed oligomorphic permutation groups) is equivalent to being able to recover $\sA$ from $\Aut(\sA)$ up to bi-interpretation. And there are well-established (though not fully general) methods for proving properties that imply automatic homeomorphicity of $\Aut(\sA)$ with respect to the class $\cgroup$ of closed subgroups of $S_\Omega$. 
%We summarise some such results in Table~\ref{table:AHforgroups}. 
A notable fact for groups is that by the Open Mapping Theorem for Polish groups~\cite[Theorem 1.2.6]{BeckerKechris}, 
if $\Aut(\sA)$ has automatic continuity with respect to a given class of Polish groups, then it also has automatic homeomorphicity with respect to it; in particular, this applies to $\cgroup$. %So, if a topological group $G$ has automatic continuity with respect to $\cgroup$, it also has automatic homeomorphicity with respect to it. 
Automatic continuity with respect to $\cgroup$ can be equivalently characterised by the \textbf{small index property} (SIP): a Polish group has the small index property if every subgroup of index $\leq\aleph_0$ is open.\footnote{{Some authors define the SIP by asking that subgroups of index $<2^{\aleph_0}$ are open~\cite{MacphersonSurvey}. By \mic{the Baire category theorem}, every open subgroup of a Polish group has index $\leq\aleph_0$, which is why we use the former definition for the purposes of automatic continuity.}} 
%(by Baire category, every open subgroup of a Polish group has index $\leq\aleph_0$). 
The small index property has been proved for several natural classes of $\omega$-categorical structures: $S_\Omega$, $\Aut(\pau{\mathbb{Q}; <)}$, the automorphism group of the countable atomless Boolean algebra, the automorphism groups of all $\omega$-categorical $\omega$-stable structures and of the random graph~\cite{DixonNeumannThomas, Truss, kwiatkowska2012group, HodgesHodkinsonLascarShelah}. With the exception of $\pau{\Aut(\mathbb{Q}; <)}$, the mentioned groups have been shown to have a stronger property than SIP, known as  \textbf{ample generics}, which implies automatic continuity with respect to the class $\csecondgroup$ of second countable topological groups~\cite{KechrisRosendal}. Finally, showing that $\Aut(\sA)$ has a \textbf{weak $\forall\exists$-interpretation} (in the sense of~\cite{Rubin}) also implies that $\Aut(\sA)$ has automatic homeomorphicity with respect to $\cgroup$. This latter method applies to several structures for which the SIP is not known (such as the generic poset), or does not hold (such as the Cherlin-Hrushovski structure, cf.~Definition~\ref{def:CH})~\cite{MacphersonBarbina}. \\

%\remmic{Perhaps put example? Am I right that for the poset it is unknown?}\\

%\begin{table}
%\begin{center}
%\begin{tabular}[t]{|c|c|c|}
%\hline

%$\Aut(\sA)$ has AH
%wrt $\cgroup$
% & Why? & Citation\\
%\hline
%\hline
%$(\mathbb{N};=)$ & ample generics & \\
%\hline
%random (Rado) graph  $\mathcal{R}$ & ample generics & \\
%\hline
%$\omega$-categorical $\omega$-stable structures & ample generics & \\
%\hline
%generic $K_n$-free graph $\mathcal{H}_n$ & ample generics & %\\
%\hline
%countable atomless Boolean algebra $B_\infty$ & ample generics & \\
%\hline
%$(\mathbb{Q}, <)$ & small index property & \\
%\hline
%\begin{tabular}{c} countable vector / simplectic / unitary /\\
%orthogonal spaces over finite fields
%\end{tabular} & small index property & \\
%\hline
%generic poset & weak $\forall\exists$-interpretation & \\
%\hline
%generic tournament & weak $\forall\exists$-interpretation & \\
%\hline
%Cherlin-Hrushovski structure & weak $\forall\exists$-interpretation & \\
%\hline
%\end{tabular}
%\end{center}
%    \caption{List of $\omega$-categorical structures $\sA$ such that $\Aut(\sA)$ has automatic homeomorphicity with respect to $\cgroup$, with the relevant reason and citation.}\label{table:AHforgroups} 
%\end{table}

Moving from automorphism groups to monoids of elementary embeddings, endomorphism monoids, and polymorphism clones, the way various notions of reconstruction interact takes several unexpected turns. The point of this survey is to help the reader navigate this sea of notions of topological reconstruction and the techniques that have been developed in this field for spaces of symmetries other than automorphism groups. In the process of doing so, we polished several existing results in the literature to clarify the picture of the current state of the art, and we also prove the significant transfer result (Theorem~\ref{thm:mainthm}) that, for $\sA$ a countable $\omega$-categorical structure, automatic homeomorphicity of $\Aut(\sA)$ with respect to $\cgroup$ implies automatic homeomorphicity of $\EEmb(\sA)$ with respect to the class $\csemi$ of closed transformation monoids on $\Omega$.\\

\begin{figure}[t] % Add option [t] to place at top of the page
\centering

\begin{tikzpicture}[scale=1.1]

  \node (a0) at (0,0) {$\cS$ has AC wrt.};
  \node (a1) at (2,0) {$\cend$};
  \node[anchor=east] at (2.9,1) {(10)};
  \node (a2) at (4,0) {$\csemi$};
  \node (a3) at (8,0) {$\csecondsemi$};
    \node[anchor=west] at (7.2,1) {(11)};
  \draw[-implies, double, thick] (a2) to (a1);
    \draw[-implies, double, thick] (a3) to (a2);

      \node (b0) at (0,2) {$\cS$ has AH wrt.};
  \node (b1) at (2,2) {$\cend$};
  \node (b2) at (4,2) {$\csemi$};
  \node (b3) at (6,2) {\pau{$\cpolsemi$ (UPP)}};
  \draw[-implies, double, thick] (b2) to (b1);
    \draw[-implies, double, thick] (b3) to (b2);
   % \draw[-implies, double, thick] (a1) to (b1);
       \draw[middlearrow={Rays}, -implies, double, thick, dashed] (b1)  to  (a2);
   %\draw[-implies, double, thick] (a2) to (b2);
    \draw[middlearrow={Rays},-implies, double, thick, dashed] (b3) to (a3);

     \node (c0) at (0,4) {$\overline{G}$ has AH wrt.};
  \node (c1) at (2,4) {$\cend$};
    \node[anchor=east] at (2,5) {(7)};
\node[anchor=west] at (4,5) {(7)};
    
  \node (c2) at (4,4) {$\csemi$};
  \draw[-implies, double, thick] (c2) to (c1);

   \node (d0) at (0,6) {$G$ has AH wrt.};
  \node (d1) at (2,6) {$\caut$};
  \node (d2) at (4,6) {$\cgroup$};
  \node (d3) at (6,6) {\pau{$\cpolgroup$ (UPP)}};
  \draw[-implies, double, thick] (d2) to (d1);
    \draw[-implies, double, thick] (d3) to (d2);
      \draw[-implies, double, thick] (d1) to (c1);
      \draw[-implies, double, thick] (d2) to (c2);
   % \draw[middlearrow={Rays},-implies, double, thick] (b3) to (a3);

      \node (e0) at (0,8) {$G$ has AC wrt.};
  \node (e1) at (2,8) {$\caut$};
  \node (e2) at (4,8) {$\cgroup\pau{\Leftrightarrow}$SIP};
    \node (e3) at (8,8) {$\csecondgroup\pau{\Leftrightarrow}\csecondsemi$};
        \node at (8,8.3) {(2)};
       \node(f4) at (10,9) {ample generics};
        \node[anchor=east] at (9.1,8.7) {(3)};
       \node (f3) at (5.5,7) {weak $\forall\exists$-int.};
        \node[anchor=east] at (2,7) {\pau{(4)}};
         \node[anchor=west] at (4,7) {\pau{(4)}};
        \node at (5.3,7.5) {(5)};
        \node at (3.2,7) {(5)};
        \node at (5.3,6.5) {(6)};
          \node[anchor=west] at (7,4) {(8)};
            \node[anchor=west] at (9,4) {(8)};
    
  \draw[-implies, double, thick] (e2) to (e1);
        \draw[-implies, double, thick] (e3) to (e2);
         \draw[-implies, double, thick] (f4) to (e3);
      \draw[-implies, double, thick] (e1) to (d1);
      \draw[-implies, double, thick] (e2) to (d2);
      \draw[-implies, double, thick] (e3) to (d3);
        \draw[-implies, double, thick] (f3) to (d2);
        \draw[middlearrow={Rays},implies-implies, double, thick, dashed] (f3) to (e2);
         \draw[middlearrow={Rays},-implies, double, thick, dashed] (d2) to[bend left] (e2);

         \draw[middlearrow={Rays},-implies, double, thick, dashed] (f4) to[bend left] (b1);
         \draw[middlearrow={Rays},-implies, double, thick, dashed] (f4) to (a3);

\draw[-implies, double, thick, dotted] (e2)  to[bend left] node[above, midway] {? (1)} (e3);

\draw[-implies, double, thick, dotted] (a1)  to node[left, midway] {? (9)} (b1);
\draw[-implies, double, thick, dotted] (a2)  to node[right, midway] {? (9)} (b2);
%\draw[-implies, double, thick, dotted] (c2)  to[bend right] node[right, midway] {?} (d2);
\node[align=left, anchor=west] at (-1,-2) {\scriptsize (1) \cite[Question 4.7]{EJMMMP-zariski};\\
\scriptsize (2) \cite[Proposition 4.1]{EJMMMP-zariski};\\
\scriptsize (3) \cite[Theorem 6.24]{KechrisRosendal};\\
\scriptsize \pau{(4) Open Mapping Theorem for Polish groups;}\\
\scriptsize (5) \cite[Theorem 4.4]{MacphersonBarbina};\\
\scriptsize (6) \cite[Proposition 2.9]{Lascar};\\
};
\node[align=left, anchor=west] at (5.5,-2) {\scriptsize(7) Theorem~\ref{thm:mainthm};\\
\scriptsize (8) Theorem~\ref{thm:ACdoesnotimplyAH}\\
\scriptsize (9) Question~\ref{q:acandah}; \\
\scriptsize (10) \cite[Corollary 10]{Reconstruction};\\
\scriptsize (11) \cite[Proposition 2.3]{pinsker2023semigroup};\\
};
\end{tikzpicture}

\caption{Implications amongst automatic continuity (AC) and automatic homeomorphicity (AH) notions with respect to different classes, with the relevant citations. Above, $G$ indicates the automorphism group of an $\omega$-categorical structure, $\overline{G}$ its monoid of elementary embeddings, and $\cS$ its endomorphism monoid.}
\label{fig:AHandAC}
\end{figure}
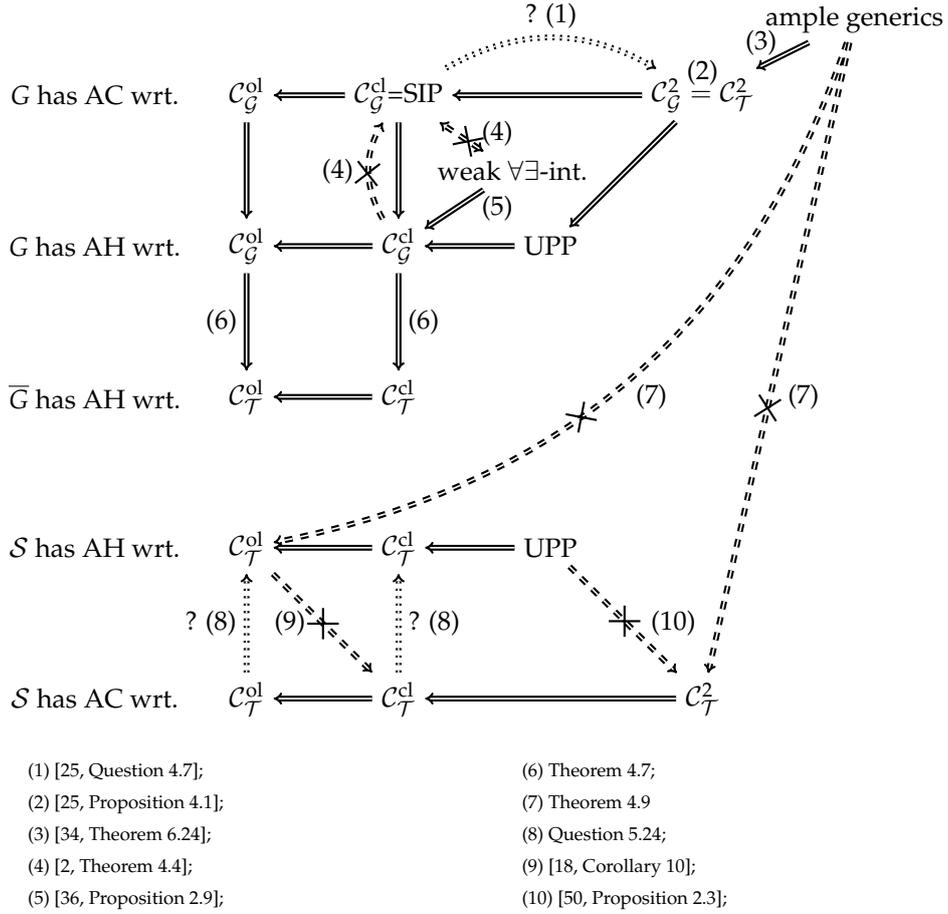

Roughly, the picture we will paint for monoids of elementary embeddings and endomorphism monoids is as follows: very little is known about reconstructibility (wrt $\csemi$) as a separate notion from automatic homeomorphicity (wrt $\csemi$). The only known $\omega$-categorical  counterexample to reconstruction is the same for automorphism groups, monoids of elementary embeddings, endomorphism monoids, and polymorphism clones~\cite{BodirskyEvansKompatscherPinsker}, and comes from a classical construction by~\cite{EvansHewitt}. Regarding automatic homeomorphicity and automatic continuity, we summarise the current situation in Figure~\ref{fig:AHandAC}. Notably, there is no Open Mapping Theorem for Polish semigroups, and hence automatic continuity and homeomorphicity are formally incomparable notions \mic{(and actually distinct~\cite{EJMMMP-zariski}).} The high connectivity of the top part of the figure displays the main result of this paper (Theorem~\ref{thm:mainthm}). Nevertheless, the reader should not be fooled by what is omitted from the figure: no monoid of  elementary embeddings of an $\omega$-categorical structure has UPP or automatic continuity with respect to $\csemi$. Meanwhile, the bottom of the figure is highly disconnected from the top half, with the only implications being the obvious ones. In particular, no matter how strong the reconstruction properties of $\Aut(\sA)$ are, $\End(\sA)$ may fail to have automatic homeomorphicity even with respect to itself. This is the result of our efforts in polishing previous results of~\cite{Reconstruction}.\\

Despite the disconnect  between the bottom and top part of Figure~\ref{fig:AHandAC}, several techniques have been developed to prove versions of automatic homeomorphicity for endomorphism monoids (and polymorphism clones). \mic{On one hand, there is a body of 
%``bottom-up'' 
%
results which rely on lifting 
%transferring 
%upwards 
automatic homeomorphicity from $\Aut(\sA)$ to $\End(\sA)$, passing through $\EEmb(\sA)$.} This in particular includes work of~\cite{behrisch-truss-vargas, truss2021reconstructing, behrisch-vargasgarcia-stronger, EJMMMP-zariski} and previous work on monoids of elementary embeddings now subsumed by Theorem~\ref{thm:mainthm}~\cite{pech-saturated, Reconstruction, behrisch-truss-vargas, truss2021reconstructing}. On the other hand, the seminal paper~\cite{EJMMMP-zariski} introduced several new ideas and techniques which are especially suited to the study of semigroups and clones and which do not  (necessarily) rely on upwards transfer from the study of some underlying topological group. In particular, the authors introduce \textbf{Property X}, a frequently occurring property for semigroups that can be used to deduce maximality of the topology of pointwise convergence amongst Polish semigroup topologies. Property X can then be mixed with the study of the \textbf{Zariski topology} $\zar$, a topology defined through purely intrinsic information on an algebraic semigroup, which, if shown to coincide with $\pw$, implies that $\pw$ is minimal amongst Hausdorff semigroup topologies. This yields a route to prove UPP without relying on any form of reconstruction for an underlying group which is often fruitful since $\pw=\zar$ on $\End(\sA)$ under fairly general conditions~\cite{PINSKER_SCHINDLER_2023, EJMPP, new}. The latter fact is notable since it stands in stark contrast with the situation for groups, where the Zariski topology  tends to be less well-behaved and usually distinct from $\pw$~\cite{ghadernezhad2024group, bardyla2025note}.\\

Indeed, the work of~\cite{PINSKER_SCHINDLER_2023} in combination with~\cite{Rubin} (see~\cite{RomanMichael}) yields another remarkable fact that we note in this survey: monoids of elementary embeddings of $\omega$-categorical structures with no algebraicity have \textbf{automatic action compatibility} in the sense that any algebraic isomorphism between two of them is induced from a bijection between their underlying domains. The analogue of this statement for automorphism groups 
%of $\omega$-categorical structures with no algebraicity 
is wide open, with substantial works in the field proving it under additional assumptions %special cases
~\cite{Rubin, PaoliniShelahReconstructing, MacphersonBarbina}.\\

In spite of the growing body of available techniques, we are still far from any form of complete picture %or sufficiently general techniques 
for topological reconstruction of endomorphism monoids or polymorphism clones. Throughout this guide, we include several open questions which we believe may help the field grow.

%We write $\pw$ for the topology of pointwise convergence.

\section{Preliminaries}
We will assume some basic knowledge of model theory, such as the first five chapters of~\cite{Tent-Ziegler}. For more on $\omega$-categorical structures and topological reconstruction in the context of automorphism groups, we refer the reader to the survey~\cite{MacphersonSurvey}.

\subsection{Spaces of symmetries and their topologies}\label{sec:spacesofsymm}

Let $\Omega$ be a countable set.  We endow $\Omega^\Omega$ with the topology of pointwise convergence $\pw$ given by the product topology on $\Omega^\Omega$, where $\Omega$ is assigned the discrete topology. We write $\fullclone^{(1)}:=\Omega^\Omega$ for the \textbf{full transformation monoid} on $\Omega$, where the monoid operation is composition of functions. When we talk about the topology of pointwise convergence $\pw$ on any subspace of $\fullclone^{(1)}$ we mean the induced subspace topology. The monoid $\fullclone^{(1)}$ with the topology $\pw$ is a \textbf{topological semigroup} in the sense that the semigroup operation is continuous. Moreover, $\fullclone^{(1)}$ is Polish, i.e.~separable and completely metrizable. Closed submonoids of $\fullclone^{(1)}$ are still topological semigroups with $\pw$. The full symmetric group $S_\Omega$ with $\pw$ is a  Polish \textbf{topological group} (i.e.~%in the sense that 
composition and inversion are continuous), and any of its %closed 
subgroups which is closed in $S_\Omega$ is a Polish topological group. 

For $n\geq 1$, let $\fullclone^{(n)}=\Omega^{{\Omega}^n}$ be the set of functions $f:\Omega^n\to \Omega$, and $\fullclone:=\bigcup_{n\in\mathbb{N}} \fullclone^{(n)}$. We call $\fullclone$ the \textbf{full clone} on $\Omega$. Analogously to the previous paragraph, we endow each $\fullclone^{(n)}$ with the product topology and endow $\fullclone$ with the sum topology.
%We endow $\fullclone$ with the topology of pointwise convergence $\pw$ by assigning each $\fullclone^{(n)}$ the product topology, where $\Omega$ was assigned the discrete topology, and then taking the sum topology for $\fullclone$. 
This defines the topology of pointwise convergence $\pw$ on $\fullclone$. Note that $\pw$ has a \mic{subbasis} of open sets $\{\cU_{(\overline{a}, b)}\vert \ \overline{a}\in\Omega^n, b\in\Omega, n\geq 1\}$, where for $\overline{a}\in\Omega^{n}$, $b\in\Omega$, 
\[\cU_{(\overline{a}, b)}:=\{f\in\fullclone^{(n)}\vert \ f(\overline{a})=b\}\;.\]
Similarly, $\{\cU_{(a, b)}\vert \ a, b\in\Omega\}$ forms a subbasis of open sets of $\pw$ on $\fullclone^{(1)}$. We often find it convenient to think of the notion of topological closure induced by $\pw$: given $W\subseteq\fullclone$ and $f\in\fullclone$, we say that $f\in\overline{W}$ if for all finite $F\subseteq\Omega$ there is some $g\in W$ such that $g_{\upharpoonright F}=f_{\upharpoonright F}$. Topological closure on $\fullclone^{(1)}$ is identical.

\begin{definition}
We call ${C}\subseteq \fullclone$ a \textbf{function clone} over \mic{$\Omega$} if
\begin{itemize}
    \item ${C}$ contains all projections:
   for each $1\leq i\leq k\in\mathbb{N}$, ${C}$ contains the $k$-ary projection to the $i$th coordinate $\pi_{i}^k\in\fullclone^{(k)}$, given by $(x_1, \dots, x_k)\mapsto x_i$;
    \item ${C}$ is closed under composition: for all $f\in{C}\cap\fullclone^{(n)}$ and all $g_1, \dots, g_n\in {C}\cap\fullclone^{(m)}$, $f(g_1, \dots, g_n)$, given by
    \[(x_1, \dots, x_m)\mapsto f(g_1(x_1, \dots, x_m), \dots, g_n(x_1, \dots, x_m)),\]
    is in ${C}\cap\fullclone^{(m)}$.
\end{itemize}   
\end{definition}
We say that a function clone is \textbf{closed} if it is closed in $\pw$ on $\fullclone$. Given a set of functions $W\subseteq\fullclone$, we denote by $\cclosure{W}$ the smallest closed function clone containing $W$. A closed function clone $\cC$ together with $\pw$ is a (Polish)  \textbf{topological clone} in the sense that each $\cC\cap \fullclone^{(n)}$ is clopen and each operation $\mathrm{comp}^n_m:\fullclone^{(n)}\times (\fullclone^{(m)})^n\to \fullclone^{(m)}$ given by $(f, g_1, \dots, g_n)\mapsto f(g_1, \dots, g_n)$ is continuous. 

\begin{definition} Let $\cC$ and $\cD$ be function clones \mic{on a set $\Omega$}. A \textbf{clone homomorphism} $\xi:\cC\to\cD$ is a map which preserves arities of functions, is \mic{the identity}
%constant 
on the projections $\{\pi_i^k\vert 1\leq i\leq k\}$ and preserves composition of functions. 
\end{definition}

\begin{definition}
Let $\sA$ and $\sB$ be first-order structures in the same relational language. A map $f:\sA\to\sB$ is a \textbf{homomorphism} if it sends tuples in a relation of $\sA$ to tuples in the corresponding relation of $\sB$; it is an \textbf{embedding} if it is moreover injective and sends only tuples in a relation of $\sA$ to tuples in the corresponding relation of $\sB$; and it is an \textbf{elementary embedding} if it preserves the truth of arbitrary first-order formulas. \mic{We indicate that $\sA$ is an elementary substructure of $\sB$ (i.e.~the identity is an elementary embedding) by $\sA\preceq \sB$.}
%    Let $\sA$ and $\sB$ be first-order structures in the same relational language $\mathcal{L}$. By an \textbf{homomorphism} we mean a map $f:\sA\to\sB$ such that for each $\mathcal{L}$-relation $R$, if $\sA\vDash R(\overline{a})$, then $\sB\vDash R(f(\overline{a}))$. By an \textbf{embedding} we mean an injective map $f:\sA\to\sB$ such that for each $\mathcal{L}$-relation $R$, $\sA\vDash R(\overline{a})$ if and only if $\sB\vDash R(f(\overline{a}))$. We say that an embedding is \textbf{elementary} if for every first-order $\mathcal{L}$-formula $\phi(\overline{x})$, $\sA\vDash \phi(\overline{a})$ if and only if $\sB\vDash \phi(f(\overline{a}))$. 
 We write $\Aut(\sA)$, $\Emb(\sA)$, $\EEmb(\sA)$, and $\End(\sA)$ for, respectively, the automorphism group, the monoid of self-embeddings, the monoid of elementary self-embeddings, and the endomorphism monoid of $\sA$. 
By a \textbf{polymorphism} of $\sA$ we mean a map $\sA^n\to \sA$ for some $n\geq 1$ which preserves  each relation $R$ of $\sA$: 
%relations in the sense that, given an $\mathcal{L}$-relation $R$,
    \[\text{ if }\begin{pmatrix}
    a^1_1\\
    \vdots\\
    a^1_k
\end{pmatrix}, \dots, 
\begin{pmatrix}
    a^n_1\\
    \vdots\\
    a^n_k
\end{pmatrix}
\in R \text{, then }
\begin{pmatrix}
    f(a^1_1, \dots, a^n_1)\\
    \vdots\\
    f(a^1_k, \dots, a^n_k)
\end{pmatrix}\in R.\]
Note that unary polymorphisms correspond to endomorphisms and that the polymorphisms of a relational structure form a closed subclone of $\fullclone$. We write $\Pol(\sA)$ for the polymorphism clone of $\sA$. 
\end{definition}

The following correspondences are easy to prove (allowing relational structures to be in an infinite language):
\begin{itemize}
    \item closed subgroups of $S_\Omega$ correspond to  automorphism groups of relational structures on $\Omega$;
    \item closed submonoids of $\fullclone^{(1)}$ correspond to endomorphism monoids of relational structures on $\Omega$;
    \item closed function clones on $\Omega$ correspond to polymorphism clones of relational structures on $\Omega$. 
\end{itemize}

\subsection{\texorpdfstring{$\omega$}{omega}-categorical structures and related notions}

\begin{definition} Let $\Omega$ be a countable set (we allow $\Omega$ to be finite). We say that a permutation group $G\acts \Omega$ is \textbf{oligomorphic} if $G\acts \Omega^n$ has finitely many orbits for each $n\in\mathbb{N}$ in its natural action $g\cdot (x_1, \dots, x_n)=(g x_1, \dots, g x_n)$ for $(x_1, \dots, x_n)\in \Omega^n$. A structure $\sA$ is $\bm{\omega}$\textbf{-categorical} if its automorphism group is oligomorphic. 
\end{definition}
With our definition, all finite structures are $\omega$-categorical. A large class of countably infinite $\omega$-categorical structures is given by homogeneous structures in a finite relational language (\textbf{finitely homogeneous} structures):
%\begin{definition} %We say that
\pau{a countably infinite structure $\sA$ in a relational language is \textbf{homogeneous} if every isomorphism between finite substructures extends to an automorphism  of $\sA$.}
%\end{definition}
%The ages (i.e., classes of finite substructures) of homogeneous 
Homogeneous 
structures are in one-to-one correspondence with certain well-behaved classes of finite relational structures known as Fra\"{i}ss\'{e} classes. Given a Fra\"{i}ss\'{e} class $\cK$, the homogeneous structure whose \textbf{age} (i.e.~class of finite substructures) is $\cK$ is known as its \textbf{Fra\"{i}ss\'{e} limit}. We refer the reader to~\cite{MacphersonSurvey} for a survey on $\omega$-categorical and homogeneous structures. We usually name homogeneous structure after their age (for which they are the \textbf{generic} countable object). Thus, the generic $\mic{\mathbb{K}_n}$-free graph $\mic{\mathbb{H}_n}$ is the countable homogeneous graph whose age is the class of finite graphs which do not embed a copy of the complete graph $\mic{\mathbb{K}_n}$, the generic poset is the Fra\"{i}ss\'{e} limit of finite partially ordered sets, and so on.\\

Spaces of symmetries of $\omega$-categorical structures have natural characterisations in terms of their actions: closed oligomorphic subgroups of $S_\Omega$ correspond to the automorphism groups of $\omega$-categorical structures; closed transformation monoids on $\Omega^\Omega$ are called  oligomorphic
 if they contain an oligomorphic permutation group, and correspond to endomorphism monoids of $\omega$-categorical structures.  Moreover, monoids of elementary embeddings of $\omega$-categorical structures correspond precisely to closed oligomorphic transformation monoids with a dense set of invertibles. Similarly, polymorphism clones of $\omega$-categorical structures correspond to closed oligomorphic clones, i.e., closed clones whose unary operations are an oligomorphic transformation monoid.\\

\mic{Often, the following weakening of $\omega$-categoricity will be sufficient:}
\begin{definition}
    \pau{A countable structure $\sA$ is \textbf{saturated} if for each of its finite subsets $B$, $\sA$ realises all $1$-types over $B$.}
\end{definition}
%These are countable structures $\sA$ such that for each of their finite subsets $B$, $\sA$ realises all $1$-types over $B$.
In particular, $\omega$-categorical structures are saturated.
%These are countable structures which realise every consistent $n$-type (over $\emptyset$) in their first-order theory for each $n\in\mathbb{N}$. In particular, $\omega$-categorical structures are saturated.
%It is well-known that the first-order theories which have a countable saturated model are exactly those theories which are \textbf{small}, i.e., which, for each $n\in\mathbb{N}$ have at most countably many types of $n$-tuples over $\emptyset$~\cite{Tent-Ziegler}.\remmic{small needed?} All $\omega$-categorical structures are saturated. An example of a countable saturated, but not $\omega$-categorical structure is the graph consisting of the disjoint union of countably many countable paths. A countable saturated structure $\sA$ is strongly $\omega$-homogeneous,\remmic{needed?} in the sense that if two finite $n$-tuples $\overline{a}$ and $\overline{a}'$ have the same type, then there is an automorphism of $\sA$ sending $\overline{a}$ to $\overline{a}'$. This in particular implies that 
In countable saturated structures we have that $\EEmb(\sA)=\overline{\Aut(\sA)}$.\\

The latter fact provides one motivation for reconstructing the topology on the class %$\cclgroup$ 
 of closed submonoids $\cS$ of $\Omega^\Omega$ whose subgroup $G$ of invertibles is dense. %: in the case where $G$ is the automorphism group of a saturated structure, such $\cS$ consists of its elementary embeddings. 
 Further, monoids of this class are an essential intermediate step when lifting reconstruction results from automorphism groups to endomorphism monoids. Finally, if $G$ is even oligomorphic, such monoids are called \textbf{model-complete cores}, and they appear as an essential concept in various applications e.g.~in Ramsey theory or Constraint Satisfaction (see \cite{MottetPinskerCores, BKOPP-equations, Cores-journal, BodirskyRamsey}). The name of model-complete core was originally given to $\omega$-categorical structures whose endomorphism monoid has a dense set of invertibles; it was shown in~\cite{Cores-journal} that every $\omega$-categorical structure $\sA$ has a substructure $\sC$ which is a model-complete core and such that $\sA$ has a homomorphism into $\sC$, and $\sC$ is $\omega$-categorical and unique up to isomorphism~\cite{Cores-journal}. Closed clones whose monoid of  unary functions are a model-complete core are called \mic{\textbf{model-complete core clones}.}
 %the same.
%Let $\sA$ be an $\omega$-categorical structure. A \textbf{model-complete core  of $\sA$} is any substructure $\sC$ of $\sA$ such that there is an endomorphism of $\sA$ into $\sC$ and with the property that all endomorphisms of $\sC$ are locally automorphisms (of $\sC$ or $\sA$; the two conditions turn out  equivalent~\cite{BKOPP-equations}). Every $\sA$ has a model-complete core, and it is $\omega$-categorical and unique up to isomorphism~\cite{Cores-journal}. The structure  $\sA$ is said to have a \textbf{mobile core} if any of its elements is contained in the image of an endomorphism into a  model-complete core.
%\end{definition}
%A structure is called a model-complete core if it is isomorphic to its model-complete core; this is the case if and only if its endomorphism monoid has a dense set of invertibles. Trivially, such structures have a mobile core. {Note that every $G\in\caut$ is the automorphism group of a model-complete core by considering the canonical structure for $G$.}

Given a group $G\acts\Omega$ and  $B\subseteq \Omega$ finite, the \textbf{algebraic closure} of $B$, $\acl(B)$, is the set of all elements of $\Omega$ which have finite orbit under the action $G_{(B)}\acts\Omega$ of the pointwise stabiliser of $B$. We say that $G\acts \Omega$ has \textbf{locally finite algebraic closure} if the algebraic closure of any finite set is finite; {this is the case for all oligomorphic actions.}  Moreover, $G\acts\Omega$ has \textbf{no algebraicity} if for all finite $B\subseteq\Omega$, $\acl(B)=B$. {It is easy to see that this implies that {the center $Z(G)$} of $G$ is trivial.}\\

%\remmic{acl operator, no algebraicity missing?}
%This is sometimes known as group-theoretic algebraic closure to distinguish it from model-theoretic algebraic closure. However, in the context of countable saturated structures

\mic{Some remarkable results of~\cite{Lascar, lascar1989demi} on monoids of elementary embeddings on a structure $\sA$ (with domain $\Omega$) rely on the assumption that $\sA$ %the 
%underlying 
%group $G$ 
%\pau{structure $\sA$ (with domain $\Omega$) for which we are considering $\EEmb(\sA)$}
is $\bm{G}$\textbf{-finite}:\footnote{\mic{We note that several papers on reconstruction use a different definition of $G$-finiteness as given  in~\cite[p.35]{lascar1989demi}. However, that definition is known to be equivalent to $G$-finiteness (as needed for the results of Lascar we mention) only under the additional assumption of $G$-compactness, which can fail even in the $\omega$-categorical setting~\cite{ivanov2010countably}. This confusion arose from the fact that~\cite{lascar1989demi} tacitly assumes model-theoretic stability, which implies $G$-compactness.}}
%Several papers define $G$-finiteness as meaning that for every open subgroup $U\leq G$, the intersection of all closed subgroups of finite index in $U$ still has finite index in $U$. However, this definition agrees with $G$-finiteness (as used by Lascar) only under an additional assumption known as $G$-compactness. And we know of  $\omega$-categorical structures which are not $G$-compact~\cite{ivanov2010countably}. This confusion seems to arise from the fact that~\cite{lascar1989demi} gives the aforementioned variant of $G$-finiteness as a definition. However, in~\cite{lascar1989demi} it is assumed throughout that one is working with a stable theory for the purposes of exposition; and stable theories are $G$-compact. It is however, unclear whether the two definitions of $G$-finiteness occurring in the literature agree or not, and it would be interesting to see whether they can be distinguished or shown to be equivalent. We thank Mira Tartarotti for noticing this issue.
for every finite $X\subseteq\Omega$, there is some finite $Y\subseteq \Omega$ such that $\Aut(\sA)_{(Y)}=\Autf_{(X)}(\sA)$, where $\Autf_{(X)}(\sA)$ is the group of Lascar strong automorphisms fixing $X$:}
\begin{align*}
\Autf_{(X)}(\sA) :=&\{g\in\Aut(\sA)_{(X)}\vert \text{ there are } \sB\succeq \sA, \, X\subseteq\sB’\preceq\sB,  \text{ and } f\in\Aut(\sB)_{(\sB’)}\\ 
& \text{ such that } g=f|_{\sA}\}\:.
\end{align*}

%Some remarkable results of~\cite{Lascar, lascar1989demi} on monoids of elementary embeddings rely on the assumption that the underlying group $G$ is $\bm{G}$\textbf{-finite}: for every open subgroup $U\leq G$, the intersection of all closed subgroups of finite index in $U$ still has finite index in $U$.
%\begin{definition}\label{def:Gfinite} A countably infinite $\omega$-categorical structure $\sA$ is $\bm{G}$\textbf{-finite} if for every open subgroup $U\leq\Aut(\sA)$, the intersection of all closed subgroups of finite index in $U$ still has finite index in $U$.
%\end{definition} 
%What we give above is equivalent to the original definition of Lascar~\cite{Lascar} since all of our mentions of $G$-finiteness will be based on his results. We note that in some modern papers in model theory, the term ``$G$-finite" is used for the weaker version of the above definition, where $U=\Aut(\sA)$. We advise the reader to be cautious of this inconsistency in definitions. 
As far as we know, there are no known examples of finitely homogeneous structures which are not $G$-finite. However, allowing for an infinite language it is easy to build such structures~\cite{lascar1982category}:

\begin{definition}\label{def:CH} By the Cherlin-Hrushovski structure $\sA_{\mathrm{CH}}$ we mean the homogeneous structure obtained as the Fra\"{i}ss\'{e} limit of finite structures with, for each $n\in\mathbb{N}$, a $2n$-ary equivalence relation $E_n$ on $n$-tuples, such that each $E_n$ has at most 
two equivalence classes.
\end{definition}

The Cherlin-Hrushovski structure also does not have the small index property~\cite{Lascar}.

\subsection{Interpretations}

%\begin{itemize}
%    \item Omega-categorical  structures: endomorphisms and interpretations~\cite{BodJunker}
%\end{itemize}
One motivation for reconstructing the topology of various spaces of symmetries of an $\omega$-categorical structure is that one can then further reconstruct the structure itself up to various notions of interpretation.
%Let $\sA$ be a countable structure. One motivation for reconstructing the topology of $\End(\sA)$ and $\Pol(\sA)$ from their purely algebraic structure is that one can then further reconstruct $\sA$ itself up to various notions of interpretation.

\begin{definition}\label{def:interpretation} Let $\sA$ and $\sB$ be first-order structures %, respectively in the signatures $\mathcal{L}$ and $\mathcal{L}'$, \pau{and 
\mic{with domains $A$ and $B$}. A  \textbf{first-order interpretation (fo-interpretation)} of $\sB$ in $\sA$ is a partial surjection $I:A^d\to B$ for some $d\in\mathbb{N}$ such that for each relation $R$ of $\sB$ defined by an atomic \mic{formula} %$\mathcal{L}'$-formula 
$\phi(x_1, \dots, x_n)$ \mic{over $\mathbb B$}, the relation
\begin{multline*}
 I^{-1}(R):=\\
 \{(a_1^1, \dots, a_1^d, \dots, a_n^1, \dots, a_n^d)\in A^{nd}\ \vert \
  \sB\vDash\phi(I(a_1^1,\dots, a_1^d), \dots, I(a_n^1,\dots, a_n^d))\}   
\end{multline*}
is first-order definable in $\sA$. We say that $d$ is the \textbf{dimension} {of the} interpretation $I$. Below, we are mainly interested in stronger notions of interpretations. In particular, we say that $I$ is an
\begin{itemize}
    \item \textbf{existential interpretation} (e-interpretation) if each $I^{-1}(R)$ is definable in $\sA$ as an existential formula;
    \item \textbf{existential positive interpretation} (ep-interpretation) if each $I^{-1}(R)$ is existentially positively definable in $\sA$; 
    \item \textbf{primitive positive interpretation} (pp-interpretation) if each $I^{-1}(R)$ is primitively positively definable in $\sA$.
\end{itemize}
\end{definition}

%If $\sA$ and $\sB$ are $\omega$-categorical structures, 
It is easy to see that any interpretation $I$ of $\sB$ in $\sA$ %defines a canonical 
naturally gives rise to a continuous group homomorphism  $\Aut(I)\colon \Aut(\sA)\to\Aut(\sB)$, \mic{where, for $g\in\Aut(\sA)$ and $b\in B$, $\Aut(I)(g)(b)=Ig I^{-1}(b)$.} Moreover, for $\omega$-categorical {$\sA$ and $\sB$,} a continuous group homomorphism $f:\Aut(\sA)\to\Aut(\sB)$ is of the form $\Aut(I)$ for some interpretation $I$ if and only if its range acts %$\mathrm{Im}(f)$ has 
 oligomorphically on $B${~\cite{AhlbrandtZiegler}}. Analogous statements hold for e-interpretations with respect to embedding monoids~\cite[Theorem 3.2]{BodJunker}, for ep-interpretations with respect to endomorphism monoids~\cite[Theorem 3.6]{BodJunker}, and for pp-interpretations with respect to polymorphism clones~\cite[Proposition 22]{Topo-Birk}. Whilst this already motivates understanding when  \mic{homomorphisms} between oligomorphic spaces of symmetries are continuous, our main interest in this survey lies in topological isomorphisms. We will see below these encode bi-interpretations of varying degrees of strength.

Note that interpretations compose naturally. In particular, if $I$ is an (fo/e/ep/pp)-interpretation of $\sB$ in $\sA$ and $J$ is an (fo/e/ep/pp)-interpretation \mic{of} $\sA$ in $\sB$, we say that $\sA$ and $\sB$ are  (fo/e/ep/pp)-\textbf{bi-interpretable} if the induced partial functions $I\circ J$ and $J\circ I$ are (fo/e/ep/pp)-definable in $\sB$ and $\sA$ respectively. See e.g.~\cite{AhlbrandtZiegler, BodJunker, Topo-Birk} for details.

\begin{theorem}\label{thm:biinterpretations} Let $\sA$ and $\sB$ be $\omega$-categorical structures of size $>1$. Then:
\begin{enumerate}
    \item\label{item:automorphisms} $\sA$ and $\sB$ are bi-interpretable if and only if $\Aut(\sA)$ and $\Aut(\sB)$ are isomorphic as topological groups (Coquand, Corollary 1.4 in~\cite{AhlbrandtZiegler}).
    
    \item\label{item:embeddings} If $\Emb(\sA)$ and $\Emb(\sB)$ are isomorphic as topological monoids, then $\sA$ and $\sB$ are existentially bi-interpretable~\cite[Theorem 3.14]{BodJunker}. The converse does not hold~\cite[Example 4.4]{BodJunker}.
    \item\label{item:endomorphisms} If $\sA$ and $\sB$ are existentially positively  bi-interpretable, then $\End(\sA)$ and $\End(\sB)$ are isomorphic as topological monoids. The converse holds as long as neither of $\sA$ and $\sB$ have a constant endomorphism (by~\cite[Theorem 3.12]{BodJunker}).
    \item \label{item:pp-int} $\sA$ and $\sB$ are primitively positively bi-interpretable if and only if $\Pol(\sA)$ and $\Pol(\sB)$ are isomorphic as topological clones~\cite{Topo-Birk}.
\end{enumerate}
\end{theorem}%

As mentioned, (\ref{item:endomorphisms})  does not hold in full generality: let $\sA$ be an $\omega$-categorical structure with a constant endomorphism. Let $\sA'$ be the structure obtained by adding a new element $c$ to $\sA$ as well as unary predicates for $\{c\}$ and its complement in $\sA'$ (which is $\sA$). Then $\End(\sA)$ and $\End(\sA')$ are topologically isomorphic. However, any structure existentially positively interpretable in $\sA$ also has a constant endomorphism~\cite[Lemma 3.9]{BodJunker}, so $\sA$ and $\sA'$ cannot be bi-interpretable. Nevertheless, one can get a fully general correspondence looking at clones and using~\cite{Topo-Birk}:

\begin{lemma} %Let $\sA$ and $\sB$ be $\omega$-categorical structures of size $>1$. 
{Let $\sA$ and $\sB$ be $\omega$-categorical structures.}
Then, $\sA$ and $\sB$ are existentially positively bi-interpretable  if and only if $\langle{\End(\sA)}\rangle$ and $\langle{\End(\sB)}\rangle$ are isomorphic as  topological clones.
\end{lemma}
\begin{proof} One direction is a direct consequence of Theorem~\ref{thm:biinterpretations} (\ref{item:endomorphisms}). For the other, suppose that the two clones are isomorphic. These clones are the polymorphism clones of the structures $\sA',\sB'$ obtained by enriching $\sA,\sB$ by the 4-ary relation defined by $x=y\vee u=v$. By~(\ref{item:pp-int}), $\sA',\sB'$ are pp-bi-interpretable, and hence $\sA,\sB$ are ep-bi-interpretable. The reason constant endomorphisms do not pose a problem in this setting is that they are detected by the clone structure: an endomorphism $f$ is constant if and only if $f\circ\pi_1^2=f\circ\pi_2^2$ for the two binary projections $\pi^2_1,\pi^2_2$.
\end{proof}

 %The case where both $\sA$ and $\sB$ have a constant endomorphism is not covered in~\cite{BodJunker}, but can be dealt with separately using~\cite{Topo-Birk}: two $\omega$-categorical structures $\sA,\sB$ are existentially positively bi-interpretable if and only if $\langle{\End(\sA)}\rangle$ and $\langle{\End(\sB)}\rangle$ are topologically isomorphic clones (these clones just consist of endomorphisms composed with projections). 
%\remmic{Reminder to myself: rewrite this} In fact, suppose there is a topological isomorphism $\xi:\End(\sA)\to\End(\sB)$. Given the an $n$-ary function $f$ and projections $\pi_1$ and $\pi_2$ to the first and second coordinate, $f\pi_1=f\pi_2$ if and only if $f$ is constant. Hence, by~\cite[Proposition 5]{Reconstruction} the natural extension of $\xi$ to the clone $\overline{\langle\End(\sA)\rangle}$ sends constant polymorphisms to constant polymorphisms and so yields a clone homomorphism $\xi':\overline{\langle\End(\sA)\rangle}\to \overline{\langle\End(\sB)\rangle}$. This is also a topological isomorphism. In particular, this tells us that there are $\sA'$ and $\sB'$ which are, respectively, existentially positively interdefinable with $\sA$ and $\sB$, and which are primitively positively bi-interpretable. Such pp bi-interpretation between $\sA'$ and $\sB'$ can then be used to build an existential positive bi-interpretation between $\sA$ and $\sB$. {Am I correct this is how it goes?}\remmic{Yes. If one adds $x=y\vee u=v$ to $\sA,\sB$, then the clones of the resulting structures will be precisely the essentially unary functions as above.}

Studying monoids of elementary embeddings is also relevant to the problem of reconstructing structures up to first-order bi-interpretability. {The following {folklore result} is a direct consequence of Proposition~\ref{prop:Lascarextension} and Theorem~\ref{thm:biinterpretations}.}
%The following folklore result is an easy consequence of~\cite{lascar1989demi}:
\begin{proposition} Let $\sA$ and $\sB$ be $\omega$-categorical structures {of size $>1$}. Then, $\sA$ and $\sB$ are first-order bi-interpretable if and only if $\EEmb(\sA)$ and $\EEmb(\sB)$ are isomorphic as topological monoids.%\remmic{Not convinced this should be here. Does it not fit better at Prop 16? In particular, omega-categoricity is not needed (when reformulated in terms of groups and closures)}
\end{proposition}
%\begin{proof} If $\Aut(\sA)$ and $\Aut(\sB)$ are topologically isomorphic, then this isomorphism lifts to a topological isomorphism of their closures (see Proposition~\ref{prop:Lascarextension}). Conversely, the restriction of a topological isomorphism $\EEmb(\sA)\to\EEmb(\sB)$ to  $\Aut(\sA)$ clearly is a topological isomorphism $\Aut(\sA)\to \Aut(\sB)$.%given a topological isomorphism $\eta:\EEmb(\sA)\to\EEmb(\sB)$, $\eta_{\upharpoonright \Aut(\sA)}$ is a topological isomorphism between $\Aut(\sA)$ and $\Aut(\sB)$.
%\end{proof}

{While this characterization of first-order bi-interpretability might seem less natural than the one with automorphism groups, it may actually be easier to verify than the latter.} 
%{This characterization of first-order bi-interpretability can be easier to verify than the one in terms of automorphism groups.} 
%The lemma above will be of interest to us since it may 
%it may actually be easier to prove the existence of a topological isomorphism between isomorphic monoids of elementary embeddings than  the existence of such  isomorphism between automorphism groups. 
For example, by Theorem~\ref{thm:zariskipinskschind}~\cite{PINSKER_SCHINDLER_2023}, if two $\omega$-categorical structures with no algebraicity have isomorphic endomorphism monoids of elementary embeddings, then this isomorphism will be a topological one. %; %hence,  by the above, they are first-order bi-interpretable. 
 No analogue of this result is known for automorphism groups.

\subsection{Notions of reconstruction}

\begin{definition} Let $\cS$ be a topological semigroup and $\classC$ be a class of topological semigroups. We say that
\begin{enumerate}
    \item $\cS$ is \textbf{reconstructible} with respect to $\classC$ if {whenever there is a semigroup isomorphism from $\cS$ to a member $\cT$ of $\classC$, there is also a semigroup isomorphism which is moreover a homeomorphism between $\cS$ and $\cT$;}
    %for every $\cT\in\classC$, if there is a semigroup  isomorphism between $\cS$ and $\cT$, then there is also a semigroup isomorphism which is moreover a homeomorphism between them;
    \item $\cS$ has \textbf{automatic homeomorphicity} with respect to $\classC$ if every semigroup  isomorphism between $\cS$ and a member $\cT$ of $\classC$ is a homeomorphism;
    \item $\cS$ has \textbf{automatic continuity} with respect to $\classC$ if every semigroup   homomorphism from $\cS$ to a member $\cT$ of $\classC$ is continuous.
\end{enumerate}
Analogous definitions %notions %can be defined 
exist for %with respect to 
topological groups and clones and the corresponding notions of homomorphism. %In particular, when talking about reconstructibility, automatic homeomorphicity, or automatic continuity with respect to a class of topological groups or clones, we intend the relevant notion of homomorphism to be, respectively, a group or a clone homomorphism.
\end{definition}

\begin{definition}
We define several classes of topological groups, semigroups and clones $\classC_X^Y$ depending on the parameters $X$ and $Y$ as follows: $X\in\{\cG, \overline{\cG}, \cT, \cC\}$ indicates whether the class we are looking at is a class of topological groups ($\cG$), topological monoids with a dense group of invertible elements ($\overline{\cG}$), semigroups ($\cT$), or clones $(\cC)$. Meanwhile, we use {$Y\in\{2, \mic{\mathrm{Po}}, \mathrm{cl}, \mathrm{sat}, \mathrm{ol}\}$} 
%$Y\in\{\mathrm{ol}, \mathrm{sat},  \mathrm{cl}, 2\}$
to indicate whether we are looking at the class of (in the order of decreasing generality):{
\begin{enumerate}
        \item second countable topological groups, monoids with dense group of invertibles, semigroups, clones, which we denote,  respectively, by $\csecondgroup, \mic{\csecondeemb}, \csecondsemi, \csecondclone$.
        \item\label{item:UPP} \mic{Polish groups, semigroups, or clones, which we denote $\cpolgroup, \mic{\cpoleemb}, \cpolsemi,$ and $\cpolclone$ respectively;}
\item\label{item:closedobjects} 
            closed subgroups of $S_\Omega$, closed  submonoids of $\omom$  with a dense set of invertibles, closed  submonoids of $\omom$, or closed subclones of  $\fullclone$, which we denote, respectively, by $\cgroup, \cclgroup, \csemi, \cclone$. 
            \item 
            the classes in item~\ref{item:closedobjects}   restricted to those elements whose underlying group of invertible elements is the automorphism group of a countable saturated structure. We denote these by $\csatgroup, \csateemb, \csatsemi, \csatclone$, respectively.
            \item 
            the classes in item~\ref{item:closedobjects}   restricted to those elements whose underlying group of invertible elements is oligomorphic. We denote these by $\caut, \ceemb, \cend, \cpol$, respectively.
%    \item closed oligomorphic subgroups of $\somega$, closed oligomorphic submonoids of $\omom$  with a dense set of invertibles, closed oligomorphic submonoids of $\omom$, or closed oligomorphic function clones, which we denote, respectively, by $\caut, {\ceemb}, \cend, \cpol$.% We also denote by $\ceemb$ the class of closures in $\Omega^\Omega$ of oligomorphic permutation groups;
    %\item 
    %closed subgroups of $S_\Omega$, closed submonoids of $\Omega^\Omega$ with a dense set of invertibles, closed submonoids of $\Omega^\Omega$,  or closed subclones of $\fullclone$ whose underlying group of invertible elements is the automorphism group of a countable saturated structure. We denote these, respectively, by $\csatgroup, \csateemb, \csatsemi, \csatclone$.
    %closed subgroups of $S_\Omega$, closures in $\Omega^\Omega$ of subgroups of $S_\Omega$, closed transformation monoids on $\Omega^\Omega$, or closed subclones of $\fullclone$ whose underlying group of invertible elements is the automorphism group of a countable saturated structure. We denote these, respectively, by $\csatgroup, \csateemb, \csatsemi, \csatclone$.
\end{enumerate}
}

%\begin{itemize}
  %  \item $\caut$ denotes the class of  automorphism groups of $\omega$-categorical structures with the topology of pointwise convergence;
 %   \item $\cgroup$ denotes the class of closed subgroups of $S_\infty$;
%    \item $\csecondgroup$ denotes the class of second countable topological groups;
 %   \item $\ceemb$ denotes the class of monoids of elementary embeddings of $\omega$-categorical structures;
 %   \item $\cend$ denotes the class of endomorphism monoids of $\omega$-categorical structures;
    %\item $\cC_{\overline{\mathcal{G}}}$ denotes the class of transformation monoids which are the closure (in $\fullclone^{(1)}$) of a closed subgroups of $S_\infty$.
 %   \item $\csemi$ denotes the class of closed transformation monoids on a countable set;
%    \item $\csecondsemi$ denotes the class of second countable topological semigroups;
 %   \item $\cpol$ denotes the class of polymorphism clones of $\omega$-categorical structures.
%\end{itemize}
\mic{We note that a Polish semigroup (group, clone) having  automatic homeomorphicity with respect to $\cpolsemi$ (respectively, $\cpolgroup, \cpolclone$) is equivalent to it having a unique Polish compatible topology, and so this property is commonly referred to as the \textbf{Unique Polish Property} (UPP).}
%Another established  notion of reconstruction is when a semigroup {(group, clone)}  has a unique Polish {compatible} %semigroup 
 %topology. This is known as the \textbf{Unique Polish Property} (UPP). Note that UPP is equivalent to automatic homeomorphicity with respect to the class of Polish topological semigroups {(groups, clones)}.
\end{definition}
%\remmic{As far as I can see the concept of ``dense set of invertibles'' appears here for the first time, though unmotivated. Should be use the opportunity to call such monoids/clones mc cores, and discuss that: they are an important step between groups and monoids/clones; and they are relevant if one considers structures up to hom-equivalence (bodirsky, existence)? Then Def 36 + what follows can be shortened in exchange.}

%\section{Topological reconstruction of monoids and clones}

\section{Reconstruction}

The following is implicit in~\cite{Lascar, lascar1989demi}:
\begin{prop}[{\cite[Proposition 11]{Reconstruction}}]\label{prop:Lascarextension}
Let $G$ and $H$ be closed permutation groups on $\Omega$. Let $\xi:G\to H$ be a continuous homomorphism.
% Let $\mathcal{T}$ and $\mathcal{T}'$ be closed transformation monoids on $\Omega$ with dense subsets of invertible elements $G$ and $G'$. Let $\xi:G\to G'$ be a continuous homomorphism.\remmic{Perhaps get rid of $T, T'$ and extend to closures? Seems more natural to me} 
Then,
\begin{enumerate}
    \item $\xi$ extends to a continuous homomorphism 
$\overline{\xi}:\tcl{G}\to\tcl{H}$;%$\overline{\xi}:\mathcal{T}\to\mathcal{T}'$;
    \item if $\xi$ is an isomorphism then $\overline{\xi}$ is an isomorphism and a homeomorphism.
\end{enumerate}
{Consequently, $G$ and $H$ are isomorphic as topological groups if and only if $\tcl{G}$ and $\tcl{H}$ are isomorphic as topological semigroups.}
\end{prop}

Proposition~\ref{prop:Lascarextension} was {an ingredient for}
%a step towards 
the following result of Lascar, which already suggests that monoids of elementary embeddings yield more information towards reconstruction than automorphism groups:
\begin{theorem}[\cite{lascar1989demi}] {Let $G\in\caut$ be $G$-finite, and let $H\in\cgroup$. Let $\eta:\overline{G}\to\overline{H}$ %for $\overline{H}\in\cclgroup$ 
be an algebraic isomorphism. Then, $\eta_{\upharpoonright G}:G\to H$ is a topological isomorphism.}
%Let $\sA$ be an $\omega$-categorical $G$-finite structure and $\sB$ be another countable structure. Let $\eta:\EEmb(\sA)\to\EEmb(\sB)$ be an algebraic isomorphism. Then  $\eta_{\upharpoonright\Aut(\sA)}:\Aut(\sA)\to\Aut(\sB)$ is a topological isomorphism.  
\end{theorem}

%\begin{theorem}[\cite{lascar1989demi}] Let $\sA$ be an $\omega$-categorical $G$-finite structure and $\sB$ be another countable structure. Let $\eta:\EEmb(\sA)\to\EEmb(\sB)$ be an algebraic isomorphism. Then  $\eta_{\upharpoonright\Aut(\sA)}:\Aut(\sA)\to\Aut(\sB)$ is a topological isomorphism.  
%\end{theorem}

Note that Proposition~\ref{prop:Lascarextension} yields the following upwards reconstruction result:

\begin{corollary} Let $G\in\cgroup$ have reconstruction with respect to $\cgroup$. Then, $\overline{G}$ has reconstruction with respect to $\cclgroup$.
%Let $\sA$ be an $\omega$-categorical structure such that $\Aut(\sA)$ has reconstruction with respect to $\caut$. Then, $\EEmb(\sA)$ has reconstruction with respect to $\ceemb$.\remmic{This needs at most saturation}\rpau{True. I would need to define a class for closures of closed permutation groups in that case.}
\end{corollary}
%\begin{proof} Take another $\omega$-categorical structure $\sB$ such that $\EEmb(\sB)$ is isomorphic to $\EEmb(\sA)$. In particular, this yields an isomorphism between $\Aut(\sB)$ and $\Aut(\sA)$. By reconstruction of $\Aut(\sA)$, there is a map $\xi:\Aut(\sA)\to\Aut(\sB)$ which is both an isomorphism and a homeomorphism. Both of these groups are dense in their respective monoids. Thus by Proposition~\ref{prop:Lascarextension}, $\xi$ extends to $\overline{\xi}:\overline{\Aut(\sA)}\to\overline{\Aut(\sB)}$ which is also an isomorphism and a homeomorphism. This proves reconstruction.
%\end{proof}
The result above also yields transfer of reconstruction from $G$ with respect to $\caut$ to $\overline{G}$ with respect to $\ceemb$. Very little is known about other forms of transfer of reconstructibility:
% (cf.~\cite[Questions 4.1 \& 4.2]{BodirskyEvansKompatscherPinsker}):
\begin{question} Let $G\in\cgroup$.
\begin{itemize}
    \item (reconstruction lifting) Does $G$ having reconstruction with respect to $\cgroup$ imply that $\overline{G}$ has reconstruction with respect to $\csemi$?
    \item (reconstruction descent) Does $\overline{G}$ having reconstruction with respect to $\csemi$ imply that $G$ has reconstruction with respect to $\cgroup$?
\end{itemize}
\end{question}
%\begin{question} Let {$G\in\cgroup$} have reconstruction with respect to $\cgroup$. Does $\overline{G}$ have reconstruction with respect to $\csemi$?
%\end{question}
%\remmic{Contrasting the upward / downward flavour of the two questions would help the reader.}
%\rpau{Good?}

%\begin{question} Let {$G\in\cgroup$}. Suppose that $\overline{G}$ has reconstruction with respect to $\csemi$. Does $G$ have reconstruction with respect to $\cgroup$? 
%\end{question}
The main issue with the first question is that we may not be able to ensure that any topological semigroup isomorphic to $\overline{G}$ is such that its subgroup of invertible elements is dense. For the second question, an issue is that algebraic isomorphisms between groups may not lift to algebraic isomorphisms of their closures. {Both questions are open even if $G$ is oligomorphic, cf.~\cite[Questions 4.1 \& 4.2]{BodirskyEvansKompatscherPinsker}.}

It is important to keep in mind throughout this survey that we do have a counterexample to reconstruction at all levels of symmetries of an $\omega$-categorical structure; {however, surprisingly it is}
%it is, however, 
 the only one we are aware of.
{
\begin{theorem}[\cite{EvansHewitt, BodirskyEvansKompatscherPinsker}]\label{thm:nonreconstruction} There are $G,H\in\caut$ such that
\begin{itemize}
\item the groups $G$ and $H$ are algebraically, but not topologically isomorphic;
\item the monoids $\tcl{G}$ and $\tcl{H}$ are algebraically, but not topologically isomorphic;
\item the clones  \mic{$\cclosure{{G}}$ and $\cclosure{{H}}$} are algebraically, but not topologically isomorphic.
\end{itemize}
%There are $\omega$-categorical structures $\sA$ and $\sB$ such that
%\begin{itemize}
%\item $\Aut(\sA)$ and $\Aut(\sB)$ are algebraically, but not topologically isomorphic;
%\item $\EEmb(\sA)=\End(\sA)$ and $\EEmb(\sB)=\End(\sB)$ are algebraically, but not topologically isomorphic;
%\item $\Pol(\sA)=\cclosure{\End(\sA)}$ and $\Pol(\sB)=\cclosure{\End(\sB)}$ are algebraically, but not topologically isomorphic.
%\end{itemize}
\end{theorem}
}
{
The structures in Theorem~\ref{thm:nonreconstruction} rely on the existence of separable profinite groups $P$ and $Q$ which are algebraically but not topologically isomorphic~\cite{Witt}. By a generalisation of the Cherlin-Hrushovski structure, these can then be encoded as the quotients by an oligomorphic normal subgroup of two oligomorphic groups $\Sigma_P$ and $\Sigma_{Q}$~\cite[Proposition 3.4]{BodirskyEvansKompatscherPinsker}. These are then ``perturbed'' to obtain $G$ and $H$. 
We note that the latter groups have a non-trivial center (and in particular, algebraicity) and are not $G$-finite. 
}

%It is unknown whether $\Sigma_G$ and $\Sigma_{G'}$ have a topological isomorphism, so the major challenge in the proof relies in  finding ``perturbations'' of these groups, $\Delta$ and $\Gamma$ which are still oligomorphic, isomorphic and for which there are no topological isomorphisms. Since $\Delta$ and $\Gamma$ are not topologically isomorphic, this will also hold of their closures in both $\Omega^\Omega$ and $\mathcal{O}$, and proving that these closures are still algebraically isomorphic is the part that requires further work.\remmic{This proof sketch is relatively detailed and could be shortened} We note that  $\Delta$ and $\Gamma$ have non-trivial center and their corresponding structures are not $G$-finite. 

%We also note that the technical complications in constructing $\Delta$ and $\Gamma$ might be needed since $\Sigma_G$ and $\Sigma_{G'}$ are automorphism groups of $\omega$-categorical structures with trivial algebraicity and so their closures are topologically isomorphic by~\cite{PINSKER_SCHINDLER_2023}.

\section{Automatic homeomorphicity}
\subsection{Automatic homeomorphicity for monoids of elementary embeddings}

In this subsection, we prove that for any  countable saturated structure, automatic homeomorphicity lifts from its automorphism group (with respect to $\cgroup$) to its  monoid of elementary embeddings (with respect to $\csemi$). %From the results of Lascar mentioned in the previous section, we also get automatic homeomorphicity of $\EEmb(\sA)$ (wrt $\csemi$) whenever $\sA$ is an $\omega$-categorical $G$-finite structure. 
This is an original result, completing the answer to a problem which is central to several papers on topological reconstruction for monoids~\cite{Reconstruction, behrisch-truss-vargas, truss2021reconstructing, pech-saturated}. In particular,~\cite{pech-saturated} answer this problem positively in the context of saturated structures whose automorphism group has a  trivial centre. Uncovering the key idea behind their proof is the essential tool to answer this question in full generality.

%\mic{We start by surveying some results on the lifting of automatic homeomorphicity from an automorphism group to its closure, as well as on automatic homeomorphicity of this closure under the assumption of $G$-finiteness on the group. These results were obtained in the situation where the group has a non-trivial center, and assumption that we shall remove completely.}
We begin from the following key lemma for reconstruction: 

\begin{lemma}[{{\cite[Lemma 12]{Reconstruction}}}]\label{lem:reconstructionlemma} Let $G\in\cgroup$ have automatic homeomorphicity with respect to $\cgroup$. Suppose that the only injective $\Phi\in\End(\overline{G})$ that fixes $G$ pointwise is the identity. % $\mathrm{Id}_{\overline{G}}$. 
Then $\overline{G}$ has automatic homeomorphicity with respect to $\csemi$.
\end{lemma}
The majority of this section is dedicated to proving the assumption on $\End(\overline{G})$ is always satisfied when $G$ is the automorphism group of a saturated structure.

%Below, homogeneous structures are possibly in an infinite language. 

\begin{definition}
Let $G\in\cgroup$. {We call a subset $A\subseteq \Omega$} {\textbf{free} if} the orbits of tuples under the  action $G_{\{A\}}:=\{\alpha\in G\; |\; \alpha(A)=A\}\acts A$  are precisely the restrictions of orbits of $G\acts\Omega$ to $A$. We say that $A$ is {\textbf{Galois-closed} if} the pointwise stabilizer $G_A$ of $A$ in $G$ acts without fixed points on $\Omega \setminus A$. Finally,  $f\in\tcl{G}$ is free Galois-closed if its image is.
\end{definition}
We remark that \mic{the notion of a free Galois-closed function}
%this notion 
 is a translation of the notion of superhomogeneity used in~\cite{pech-saturated} from structures to permutation groups. \mic{We shall now apply some results of that paper which are formulated there in terms of homogeneous structures. To do so, note that every closed permutation group $G$ is the automorphism group of the   homogeneous relational structure  %known as the \textbf{canonical structure}, 
 which has a relation for every orbit of $G$ on finite tuples.} 
%to do so, one simply takes, for a closed subgroup $G$ of $S_\Omega$, the \textbf{canonical structure} which has a relation for every orbit of tuples under $G$. %{Below, $Z(G)$ denotes the centre of $G$.}

\begin{proposition}
[{\cite{pech-saturated}}]\label{prop:hiddeninpech} Let $G\in\cgroup$ and let $\Phi\in\End(\tcl{G})$ fix $G$ pointwise. Then for all  free Galois-closed $h\in\tcl{G}$ there exists  $g_h\in Z(G)$, such that $\Phi(h)=h\circ g_h$. 
\end{proposition}
\begin{proof} For $f\in\tcl{G}$, let 
\[f^*:=\{(\alpha, \beta)\ \vert\; \alpha, \beta\in G, \ \alpha\circ f=f\circ \beta\}\;.\]
Clearly, for any $h\in\tcl{G}$ we have  $\Phi(h)^*=h^*$.
By Corollary 3.12 and Proposition 3.13 of~\cite{pech-saturated} this implies that there is some $g_h\in Z(G)$ such that $\Phi(h)=h\circ g_h$.  
\end{proof}

\mic{We shall now consider automorphism groups of saturated structures; these are always also automorphism groups of saturated homogeneous structures in a countable language whose theory has quantifier elimination (see, e.g.,~\cite[Lemma 2.6]{pech-saturated}). % These are called \textbf{smooth} saturated structures in~\cite{pech-saturated}, and they are homogeneous. 
This allows us to apply results from~\cite{pech-saturated}, where such structures are called smooth.} 
Our key observation is that their  main results %of~\cite{pech-saturated} 
 rely on a correspondence between maps in $\End(\overline{G})$ which fix $G$ pointwise and homomorphisms $\phi:\overline{G}\to Z(G)$ mapping $G$ to $1$:

\begin{proposition}\label{pro:whatphiislike} Let $G\in\csatgroup$ and let $\Phi\in\End(\overline{G})$ fix $G$ pointwise. Then, there is a semigroup homomorphism $\phi:\overline{G}\to Z(G)$ such that $\phi(G)=\{1\}$ and for each $f\in\overline{G}$
\[\Phi(f)=\phi(f) \circ f.\]
\end{proposition}
\begin{proof} %By Lemma 2.6 in~\cite{pech-saturated}, we can assume $G=\Aut(\sA)$ where $\sA$ is a smooth structure. In particular, 
%\pau{We shall make use of \pau{by~\cite[Propositions 4.1 and 5.1]{pech-saturated}, recalling that our notion of free Galois-closed is just a translation of the notion of the notion of superhomogeneity}

For any $f\in\tcl{G}$,
%$f\in \Emb(\sA)=\overline{G}$, 
%by Proposition~\ref{prop:findprops}, 
\mic{by~\cite[Lemma 2.6 and Proposition 5.1]{pech-saturated}}, we can take free Galois-closed $g,h\in \overline{G}$ %with Property S 
 such that $g=h\circ f$. By {Proposition~\ref{prop:hiddeninpech}}, there are $z_g, z_h\in Z(G)$ such that $\Phi(g)=g\circ z_g$ and $\Phi(h)=h\circ z_h$. Now, since $\Phi$ is a homomorphism, 
 %we get that
\[h\circ f\circ z_g=g\circ z_g=\Phi(g)=\Phi(h\circ f)=\Phi(h)\Phi(f)=h\circ z_h \circ \Phi(f).\]
Since $h$ is injective and $z_h, z_g\in Z(\overline{G})$ (being in $Z(G)$), we get
\[\Phi(f)=z_h^{-1} \circ z_g \circ f\;.\]
    In particular, we can set $\phi(f)=z_h^{-1} \circ z_g$. Since $\Phi$ fixes $G$ pointwise, we must have that $\phi(G)=1$. We now claim that $\phi:\overline{G}\to Z(G)$ is a homomorphism. To see this, note that for $f,g\in\overline{G}$, since $\phi(f), \phi(g)\in Z(G)$, 
    \[(f\circ g )\circ (\phi(f)\circ \phi(g))= (\phi(f)\circ f)\circ (\phi(g)\circ g)=\Phi(f\circ g)=(f \circ g) \circ (\phi(f\circ g))\;,\]
    which by injectivity of $f\circ g$ implies $\phi(f\circ g)=\phi(f)\circ \phi(g)$ as required.%, yielding that $\phi$ is a homomorphism. 
\end{proof}

The proof of the following lemma will use some basic model-theoretic techniques.
%{The following lemma requires some basics of model theory, such as the first four chapters of~\cite{Tent-Ziegler}:}

\begin{lemma}\label{lemma:absorbingendomorphism} Let $G\in\csatgroup$. Let $h\in \tcl{G}$. Then there are $f\in\tcl{G}$ and $\alpha,\beta\in G$ such that %Let $\sA$ be a countable saturated structure. Let $h\in\EEmb(\sA)$. Then, there are $f\in\EEmb(\sA)$ and $\alpha, \beta\in\Aut(\sA)$ such that 
\[\alpha \circ f \circ h=f\circ \beta.\]
\end{lemma}
\begin{proof} Let $h\in \tcl{G}$. {Let $\sA$ be a saturated structure with automorphism group $G$.}
%Let $\sA$ be the (saturated) orbit structure of $G$. 
Write $\sA_1$ for $\sA$ and $\sA_0$ for $h(\sA_1)$. Since $h$ is an elementary self-embedding, its inverse is an isomorphism $\alpha_0:\sA_0\to\sA_1$. %\mic{By virtue of this isomorphism, $\sA_1$ is again an elementary substructure of a structure $\sA_2$ such that $(\sA_2;\sA_1)$ (the expansion of $\sA_2$ by a predicate for the subset $\sA_1$) and $(\sA_1;\sA_0)$ are isomorphic. We construct an extension of $\alpha_0$ to an automorphism of $\sA_2$; that is, we  define by back-and-forth an isomorphism from $\sA_1\setminus \sA_0$ to $\sA_2\setminus \sA_1$ compatible with $\alpha_0$.}
We first construct an elementary extension $\sA_2\succeq \sA_1$ together with an isomorphism $\alpha_1:\sA_1\to \sA_2$ extending $\alpha_0$. To do so, take an enumeration of $\sA_1$. % (and the induced enumeration of $\sA_0$). 
Consider the type  $p(z;A_0):=\mathrm{tp}(\sA_1/\sA_0)$ in a variable $z$ of length $\omega$. %This type can be realised. 
Let $\sA_2$ be a realisation of $p(x;\alpha_0(\sA_0))$ and define $\alpha_1$ to be the map sending every element of $\sA_1$ to its realisation in $\sA_2$. This map is an isomorphism and extends $\alpha_0$. % We need to make sure that $\sA_2\succeq\sA_1$. To see this,\remmic{I don't think this is necessary as $\alpha_1$ is an isomorphism by definition which sends A0 to A1} we use Tarski's test~\cite[Theorem 2.1.2]{Tent-Ziegler}. Let $B\subseteq \sA_1$ and suppose that $\sA_2\vDash \exists x\phi(x;B)$ for some first-order formula $\phi(x;y)$ (where $x$ and $y$ indicate tuples of variables). Note that by definition of $\sA_2$ and $\alpha_1$, $\sA_2\vDash \exists x\phi(x;B)$ if and only if $\sA_1\vDash \exists x \phi(x;\alpha_1^{-1}(B))$. Since $B\subseteq \sA_1$, $\alpha_1^{-1}(B)=\alpha_0^{-1}(B)\subseteq \sA_0$. So $\sA_1\vDash \exists x \phi(x;\alpha_0^{-1}(B))$. Being the inverse of an isomorphism, $\alpha_0^{-1}$ is an elementary embedding of $\sA_1$ into $\sA_0$. Hence,  $\mathrm{tp}_{\sA_1}(B)=\mathrm{tp}_{\sA_0}(\alpha_0^{-1}(B))=\mathrm{tp}_{\sA_1}(\alpha_0^{-1}(B))$, where the last equality follows since $A_0\preceq A_1$ {and $\mathrm{tp}_{\sB}(C)$ indicates the type of the tuple $C$ in the structure $\sB$}. In particular, $\sA_1\vDash \exists x \phi(x;B)$ if and only if $\sA_1\vDash \exists x \phi(x;\alpha_0^{-1}(B))$. But since the latter is implied by $\sA_2\vDash \exists x\phi(x;B)$, we have that $\sA_1\vDash \exists x \phi(x;B)$, completing Tarski's test and proving that $\sA_1\preceq\sA_2$ (through the elementary embedding $\alpha_1^{-1}$). 
By repeating this process, we get an elementary chain
\[\sA_0\preceq \sA_1\preceq\dots\preceq\sA_i\preceq \sA_{i+1}\preceq\dots \]
where, at each stage $i\geq 1$ we have an isomorphism {$\alpha_{i}: \sA_{i}\to\sA_{i+1}$} extending $\alpha_{i-1}$. Write $\sA'$ for the union of this elementary chain and $\alpha$ for $\bigcup_{i\in\mathbb{N}}\alpha_i$. By construction, $\alpha$ is an automorphism of $\sA'$. Since it is a union of a countable chain of countable saturated models, $\sA'$ is still a countable saturated model of the first-order theory of $\sA$. In particular, there is an isomorphism $f:\sA'\to \sA_1$ (cf.~\cite[Lemma 5.2.8]{Tent-Ziegler}). Define $h':\sA'\to\sA'$ as $f^{-1}hf$. Note that, for $a,b\in\sA'$,
\[
    h'(a)=b   \Leftrightarrow f^{-1}hf(a)=b \Leftrightarrow hf(a)= f(b)\;.\]
Hence, since $f$ is an isomorphism between $\sA'$ and $\sA$, it yields an isomorphism  between the structures $(\sA';h')\cong(\sA;h)$ (i.e., the expansions of $\sA'$ and $\sA$ by, respectively, $h'$ and $h$). In particular, to prove our lemma, it is sufficient to prove its conclusion for $(\sA';h')$ rather than for $(\sA, h)$. Now, $fh'=hf$ yields an isomorphism $\sA'\to\sA_0$. Moreover, recall that $\alpha$ is an isomorphism $\sA'\to\sA'$ such that $\alpha(\sA_0)=\alpha_0(\sA_0)=\sA_1$. Thus, since $fh'$, $\alpha$, and $f^{-1}$ are all isomorphisms, $f^{-1}\alpha f h':=\beta\in\Aut(\sA')$. 
%\pau{Hence, $\alpha \circ f\circ h'=f\circ \beta$, as desired.}
Hence,
\[\alpha \circ f\circ h'=f\circ \beta,\]
as desired. 
\end{proof}

\begin{corollary}\label{cor:nophi} Let $G\in\csatgroup$, and let $H$ be any group. Then, any homomorphism $\phi:\overline{G}\to H$ such that $\phi(G)=\{1\}$ is constant.  
\end{corollary}
\begin{proof} Let $h\in\overline{G}$. By Lemma~\ref{lemma:absorbingendomorphism} there are $f\in\overline{G}$ and $\alpha, \beta\in G$ such that
\[\alpha\circ f\circ h=f\circ \beta.\]
Since $\phi$ is a homomorphism, 
\[\phi(\alpha)\circ \phi(f)\circ \phi(h)=\phi(f)\circ \phi(\beta).\]
Since $\phi(\alpha)=\phi(\beta)=1$ and $\phi(f)\in H$, this implies that $\phi(h)=1$.
\end{proof}

\begin{theorem}\label{thm:mainthm} Let $G\in\csatgroup$ have automatic homeomorphicity with respect to $\cgroup$. Then, $\overline{G}$ has automatic homeomorphicity with respect to $\csemi$.\\
Moreover, suppose $\theta:\tcl{G}\to\Omega^\Omega$ is an injective  homomorphism and that $\theta(G)$ is closed in $S_\Omega$. Then $\theta$ is a homeomorphism onto its image, which is an element of  $\cclgroup$.
%Moreover, if $G$ has automatic homeomorphicity with respect to $\caut$, then $\overline{G}$ has automatic homeomorphicity with respect to $\cend$.\remmic{The second statement here seems redundant. One could put any pair of classes (A,B) such that a closed transformation monoid satisfies B iff its invertibles satisfy A}{PAU: technically non-redundant because AH wrt $\caut$ is weaker. But happy to remove.}\remmic{I see that they don't imply each other but it's too arbitrary}
\end{theorem}
\begin{proof}  By Proposition~\ref{pro:whatphiislike} and Corollary~\ref{cor:nophi}, any $\Phi\in\End(\overline{G})$ that fixes every element of $G$ is the identity function on $\tcl{G}$.
%$\mathrm{Id}_{\overline{G}}$. 
Hence $\overline{G}$ has automatic homeomorphicity with respect to $\csemi$ by Lemma~\ref{lem:reconstructionlemma}. {Given the first part, the proof of the second part is identical to that of~\cite[Lemma 4.1]{behrisch-truss-vargas}.}
\end{proof}
Note that the same proof yields that if $G$ has automatic homeomorphicity with respect to a subclass $\mathcal C$ of $\cgroup$ (e.g.~$\caut$), then $\tcl{G}$ has automatic homeomorphicity with respect to the subclass of those monoids in $\csemi$ whose {group of  invertibles belongs to $\mathcal C$ (e.g.~$\cend$).} 
%groups of  invertibles belong to $\mathcal C$ (e.g.~$\cend$). 
The corollary below now follows by the same argument as in~\cite[Theorem 2.4]{pech-saturated}:
\begin{corollary} Let $G\in\caut$ be $G$-finite. Then $\overline{G}$ has automatic homeomorphicity with respect to $\cclgroup$.
\end{corollary}

\subsection{Automatic homeomorphicity for endomorphism monoids}

Surprisingly, automatic homeomorphicity may fail for the endomorphism monoids of structures whose automorphism groups have automatic homeomorphicity even in very tame cases. The following is an adaptation of~\cite[Theorem 8]{Reconstruction}. Our main contribution is that the structure in the counterexample below is extremely tame. Firstly, it is model-theoretically tame, being $\omega$-stable and a reduct of a finitely homogeneous structure. These structures are very well understood~\cite{CherlinHarringtonLachlan}. Secondly, being $\omega$-categorical and $\omega$-stable, its automorphism group has ample generics~\cite{HodgesHodkinsonLascarShelah}, and so our example is also  tame from the point of view of reconstruction of its automorphism group.

\begin{theorem}\label{thm:ACdoesnotimplyAH}
{There is an $\omega$-categorical structure $\sA$ which is $\omega$-stable and first-order definable from a finitely homogeneous structure such that $\Aut(\sA)$ has automatic continuity with respect to $\csecondgroup$, but $\End(\sA)$ (and hence also the clone $\langle\End(\sA)\rangle$)  has a non-continuous automorphism. In particular, $\End(\sA)$ (and $\langle\End(\sA)\rangle$) do not have automatic homeomorphicity with respect to $\cend$ (and $\cpol$).}
\end{theorem}
\begin{proof}
    Let $\sB$ be the structure with an equivalence relation $E$ with infinitely many classes of size 2. Let $\sA'$ consist of two disjoint copies $\sB_1, \sB_2$ of $\sB$ plus an additional new element $\star$ which is (only) related to itself in $E$, and which has additional  unary predicates $B_1,B_2$ for the domains of $\sB_1, \sB_2$. It is easy to see that $\sA'$ is finitely homogeneous (hence, $\omega$-categorical) and $\omega$-stable. In particular, it has ample generics by~\cite{HodgesHodkinsonLascarShelah} and so $\Aut(\sA')$ has automatic continuity with respect to $\csecondgroup$. Let
    \begin{multline*}
 Q:=\{f: A'\to A' \vert \ f(B_1)\subseteq B_2\cup\{\star\},\  f(B_2\cup\{\star\})=\{\star\},\; f \text{ preserves } E\}.
\end{multline*}
Consider the closed monoid  $T:=\tcl{\Aut(\sA')}\cup Q$; its invertibles are precisely the elements of $\Aut(\sA')$.  Let %$\mathcal{T}_{<\infty}$ consist of those elements of $\mathcal T$ which miss only a finite number of $E$-classes of $B_2$ in their image, and 
$T_{\infty}$ consist of those elements of $Q$ \mic{for which the complement of their image contains an infinite number of $E$-classes in $B_2$.}
%which miss an infinite number of $E$-classes in $B_2$.  
Let $\alpha\in\Aut(\sA')$ flip the two elements of any $E$-class (and fix $\star$). 
We define the map $\eta:{T}\to {T}$ by
\[\eta(f):=\begin{cases}
    f & f\notin {T}_{\infty}\;;\\
    \alpha f  & f\in {T}_{\infty}\;.
\end{cases}\]
Using that $\alpha$ commutes with all elements of $\tcl{\Aut(\sA')}$, it is easy to see that $\eta$ is an automorphism of $T$. However, taking a sequence $(f_n)_{n\in\mathbb{N}}$ outside of $T_{\infty}$ converging to a non-constant element $g$ in $T_{\infty}$, $(\eta(f_n))_{n\in\mathbb{N}}$ converges to $g$ but $\eta(g)=\alpha g \neq g$. Hence $\eta$ is not continuous. \mic{Finally, any countable structure $\sA$  with endomorphism monoid ${T}$ will satisfy the statement of the theorem.}
\end{proof}

In spite of the above negative result, there are still  reasonable situations where one may {hope 
%to be able 
to lift automatic homeomorphicity} results from $\Aut(\sA)$ to $\End(\sA)$. In particular,~\cite{behrisch-truss-vargas, truss2021reconstructing} developed a sophisticated technique which they use to show automatic homeomorphicity for the endomorphism monoids of $(\mathbb{Q}\mic{;}\leq)$ and  some structures first-order definable in it with different automorphism groups. %its reducts (with reflexive relations), 
 They make use of the small index property for $\Aut(\sA)$ and some additional information on the structure of its open subgroups, which may be provided by some (possibly weak) form of elimination of imaginaries.\\ 

%of $\End(\mathbb{Q}, \leq)$ 
%and which they subsequently use in~\cite{truss2021reconstructing} to show automatic homeomorphicity of the endomorphism monoids of the non-trivial reducts of $(\mathbb{Q}, \leq)$. The technique they develop makes a clever use of the small index property for $\Aut(\sA)$ and requires some understanding of the open subgroups of $\Aut(\sA)$ which is facilitated by the presence of some (possibly weak or very weak) form of elimination of imaginaries in their proofs.\\

Their main idea is as follows: consider an isomorphism $\theta:\End(\sA)\to\cS\in\csemi$. This can be viewed as a faithful monoid action of $\End(\sA)$ on $\Omega$, the domain of $\cS$; in the following discussion, to distinguish this action from the original action of $\End(\sA)$, we write $A$ for the domain of the original action. The aim is to directly show that $\theta$ is a homeomorphism by studying how the induced action of $G:=\Aut(\sA)$ on $\Omega$ can extend to $\theta$. By the orbit-stabilizer theorem, for any $x\in\Omega$ there is a bijection between the orbit $X:=\mathrm{Orb}(x)$  of $x$ under  $G\acts\Omega$ and 
%$X:=\mathrm{Orb}(x)$ for $x\in \Omega$, and 
the left-cosets in $G$ of the stabilizer $G_x=\{g\in G\;\vert\; \theta(g)(x)=x\}$. Since $X$ is countable, $G_x$ has countable index. Thus, if $G$ has the SIP we know that $G_x$ is open. Since the open subgroups of $G$ are precisely those containing the pointwise stabilizer $G_B$ of a finite set $B$ under the action $G\acts A$, we can  conclude that $G_x\supseteq G_B$ for some finite $B$. With additional information on the open subgroups of $G$, however, we can draw stronger conclusions.\\

We say that $G$ has \textbf{elimination of imaginaries} if every open subgroup of $G$ is equal to $G_B$ for some finite $B$, and that it has  \textbf{weak elimination of imaginaries} if every open subgroup of $G$ is sandwiched between $G_B$ and the setwise stabiliser $G_{\{B\}}$ for some finite $B$. Several %\remmic{``Several''  is weak - is there a better statement?}\rpau{Not really: stuff with an equivalence relation does not have WEI. However, most of this strategy would even work with very weak elimination of imaginaries, for which there are no known finitely homogeneous counterexamples} 
$\omega$-categorical structures have weak elimination of imaginaries, e.g., $(\mathbb{N}, =)$, the random graph, and infinite vector spaces over finite fields. It is easy to prove that if $G$ has weak elimination of imaginaries, then for any open subgroup $H\leq G$ there is a finite algebraically closed $B\subseteq \mic{\Omega}$ and $L\leq G_{\{B\}}\acts B$ such that 
\[H=G_{(B, L)}:=\{f\in G \vert f_{\upharpoonright B}\in L\}\;.\]
Groups with the SIP and weak elimination of imaginaries are said to have the \textbf{strong small index property} (SSIP)~\cite{Oligo}.

Let us now return to analysing an isomorphism $\theta:\End(\sA)\to\cS\in\csemi$ assuming that $G$ has the SSIP. In that situation, for any $x\in \Omega$ we have $G_x=G_{(B, L)}$ as above. In particular, we can identify the orbit $X=\mathrm{Orb}(x)$ with the orbit of the pair $(B, L)$ under the natural action of $G$ where for $g\in G$, $g(B, L)=(g B, g L\mic{g^{-1}})$; here, $g L\mic{g^{-1}}$ is a subgroup of $G_{\{gB\}}\acts gB$. Doing this for all $G$-orbits, we can rewrite the action of $G$ on $\Omega$ as the action of $G$ on a disjoint union of $G$-orbits of pairs $(B_i, L_i)_{i\in I}$, where $I$ is a countable set. Using Theorem~\ref{thm:mainthm} and the proof of~\cite[Lemma 4.2]{behrisch-truss-vargas}, we can then also describe the action of $\overline{G}$ on $\Omega$: {for any $f\in\overline{G}$ and any $(gB_i, gL_i\mic{g^{-1}})$, where  $g\in G$ and $i\in I$,  we obtain $f(gB_i, gL_i)=(hgB_i, hgL_i\mic{g^{-1} h^{-1}})$ where $h\in G$ is an arbitrary element which agrees with $f$ on $gB_i$.} %$f\circ (gB_i, gL_i)=(fgB_i, fgL_i)$.}

For the final step, which requires the most work, one needs to describe how $\End(\sA)$ acts on the orbits of the action $G\acts \Omega$ induced by $\theta$ and deduce that this action yields that $\theta$ is a homeomorphism. %Whilst this cannot always be done successfully (as demonstrated by Theorem~\ref{thm:ACdoesnotimplyAH}), \pau{it already yielded substantial results:}
%various structures with a first-order definition in  $(\mathbb{Q}, \leq)$ (see~\cite{Cameron5}; the relations in the statement below are all taken to be reflexive):

\begin{theorem}[\cite{behrisch-truss-vargas}]\label{thm:reductsofq} \mic{$\End(\mathbb{Q}; \leq)$ has automatic homeomorphicity with respect to $\csemi$. Moreover, $\Pol(\mathbb{Q}; \leq)$ has automatic homeomorphicity with respect to $\cclone$.} 
%Moreover, $\Pol(\mathbb{Q}; \leq)$ has automatic homeomorphicity with respect to $\cclone$.}
%Let $\sA$ be any of: $(\mathbb{Q}, \leq)$, $(\mathbb{Q}, \mathrm{sep})$, $(\mathbb{Q}, \mathrm{betw})$, $(\mathbb{Q}, \mathrm{circ})$. %, where the latter relations have been defined to be reflexive. 
%Then, $\End(\sA)$ has automatic homeomorphicity with respect to $\csemi$. Moreover, $\Pol(\sA)$ has automatic homeomorphicity with respect to $\cclone$. 
\end{theorem}

\mic{The rough idea for the case of $\Pol(\mathbb{Q}; \leq)$ is the following: openness follows generally  from the presence of constant functions by the work of~\cite{Reconstruction} (specifically, Proposition~\ref{prop:constants} in the present paper); showing continuity is  more complex and exploits the machinery for monoids described above. %by a general result in~\cite{behrisch-truss-vargas}.
}
%\pau{For the case of $\Pol(\mathbb{Q}; \leq)$, given work of~\cite{Reconstruction} (i.e., Proposition~\ref{prop:constants}), it is sufficient to verify that any surjective homomorphism $\phi:\Pol(\mathbb{Q}; \leq)\to C$ where $C\in\cclone$ is continuous. This is obtained from the following Lemma:}
%\begin{lemma}[{\cite{behrisch-truss-vargas}}]
%\pau{Let $C, D\in\cclone$ and let $\phi:C\to D$ be a homomorphism. 
%Suppose that for every $b\in\Omega$ there
%is some unary function $h\in C$ with finite image such that $\phi(h)(b) = b$.
%Then, $\phi$ is continuous.}
%\end{lemma}

\mic{~\cite{truss2021reconstructing} use similar techniques to deduce automatic homeomorphicity of the endomorphism monoids {and polymorphism clones} of some further structures with a first-order definition in $(\mathbb{Q}; \leq)$.}

\section{On the number of Polish topologies}
In this subsection, we study Polish topologies on closed oligomorphic monoids. Whilst all monoids in $\ceemb$ have at least two Polish semigroup topologies, several interesting monoids in $\cend$ have the Unique Polish Property.

\subsection{Multiple Polish topologies on monoids of elementary embeddings}
\begin{lemma}[{\cite[Proposition 5.1]{EJMPP}}]\label{lem:atleasttwo} Let $\cS\in\csemi$ be a monoid of injective functions whose group of invertible elements is not closed in $\Omega^\Omega$. Then, there are at least two distinct Polish semigroup topologies on $\cS$. In particular, this is the case for any $\cS\in\ceemb$.
%\remmic{Even oligomorphic monoid}}\rpau{How can this be the case given the UPP results? We need injectivity of $\cS$ I think}
%$\sA$ be an $\omega$-categorical structure. There are at least two distinct Polish semigroup topologies on $\EEmb(\sA)$.\remmic{This really only uses that $G$ is not closed in the full monoid. $\omega$-categoricity of course implies this, but is never mentioned, so the punch line of the proof is missing}
\end{lemma}
\begin{proof}[Proof idea] 
Consider the topology $\mathcal{I}_4$ on the monoid of injective functions $\mathrm{Inj}(\Omega)$ generated by $\pw$ and the open sets $(W_a)_{a\in \Omega}$ of the form 
\[W_a:=\{f\in\mathrm{Inj}(\Omega) \ \vert \ a\not\in\mathrm{Im}(f)\}\;.\]
This is a Polish topology on $\mathrm{Inj}(\Omega)$~\cite[Theorem 5.15 \& Section 5.5]{EJMMMP-zariski}. It is easy to see that the induced topology $\tgcl$ on $\cS$ is a Polish semigroup topology in which the group of invertibles is closed, yielding that $\tgcl\neq\pw$.
%In particular, since $\cS$ is closed in $\mathrm{Inj}(\Omega)$, the topology $\tgcl$ induced by $\mathcal{I}_4$ on $\cS$ through the subspace topology is also a Polish semigroup topology. However, in $\tgcl$, non-surjective embeddings form an open set, and so $\Aut(\sA)$ is closed. Hence, $\tgcl$ and $\pw$ are distinct Polish topologies on $\EEmb(\sA)$.
\end{proof}

Similar ideas were used by~\cite{schindler2023semigroup} to prove, for every $n$, failure of UPP for the endomorphism monoids of the generic graph without independent sets of size $n$  % $n$-anticlique-free graph $\overline{\mathcal{H}_n}$ (i.e., the dual of $\mathcal{H}_n$) 
and of the complete multipartite graph with $\omega$-many blocks of size $n$. % $K_{\omega, n}$. 
 Whilst not all endomorphisms of these structures are  injective, there is some $m\in\mathbb{N}$ such that, for any endomorphism, the preimage of any point has size $\leq m$, and this is sufficient to build a similar topology to $\tgcl$ which is distinct from $\pw$.

\begin{definition} 
%Let $G\acts\Omega$ have locally finite algebraic closure. 
\pau{Let $G\acts\Omega$. For $A\subseteq \Omega$ infinite, define} \[\acl(A):=\bigcup\{\acl(B)\vert\;  B\subseteq A \text{ finite}\}\;.\]
\pau{We say that $\acl$ forms a \textbf{pregeometry} if it satisfies the exchange property: for all {$a,b\in \Omega$} and $F\subseteq \Omega$, if  $a\in\acl(F\cup\{b\})\setminus\acl(F)$,  then  $b\in\acl(F\cup\{a\})$.}
%we have that 
%\[{\text{if } a\in\acl(C\cup\{b\})\setminus\acl(C), \text{ then } b\in\acl(C\cup\{a\}).}\] 
\end{definition}
\pau{When $\acl$ forms a pregeometry, for $A, B\subseteq\Omega$, we can adequately define the \textbf{dimension} $\dim(B/A)$ of $B$ over $A$ as the smallest cardinal $\kappa$ such that there is $B'\subseteq \Omega$ with $|B'|=\kappa$ and $\acl(B'\cup A)=\acl(B\cup A)$ (see~\cite{white1986theory}).} 
%(we assign dimension $\aleph_0$ to infinite subsets of $\Omega$). 
Note that $\acl$ forms a pregeometry in any oligomorphic permutation group with no algebraicity, or in the automorphism group of any vector space over a finite field.

\begin{lemma}\label{lem:countablymany} Let $G\acts\Omega$. \mic{Suppose that $\acl$ is locally finite and forms a pregeometry.} 
%\pau{have locally finite algebraic closure and be such that} 
%be such that 
%$\acl$ forms a pregeometry. 
Then, there are {at least} countably many Polish semigroup topologies on $\overline{G}$.
\end{lemma}
\begin{proof} This is essentially the proof of~\cite[Theorem 5.22 (i) \& (vi)]{EJMMMP-zariski}. 
    Note that for any $f\in\overline{G}$, the image $\mathrm{Im}(f)$ is algebraically closed. %a closed subspace of $\sA$ in the sense that the algebraic closure of every finite subset $\mathrm{Im}(f)$ is still contained within $\mathrm{Im}(f)$. 
 Let $\tau_0:=\tgcl$ from the proof of  Lemma~\ref{lem:atleasttwo}.
For each $n\in\mathbb{N}$ let
    \[Q_n:=\{f\in\overline{G} \ \;\vert\; \dim(\Omega\pau{/} \mathrm{Im}(f))=n\}\;.\]
    For all $n\in\mathbb{N}$, we define  $\tau_n$ to be the topology generated by $\tau_{n-1}$ and $Q_{n-1}$.\\
    
\textbf{Claim 1:} For each $n\in\mathbb{N}$, $\tau_n$ is a semigroup topology.
\begin{proof}[Proof of Claim 1] We need to prove that composition is continuous. 
%We first prove that for each $l<n$, the preimage of $F_l$ under composition is open. 
Note that for all $f,g\in\overline{G}$,
\[\dim(\Omega\pau{/}\mathrm{Im}(fg))=\dim(\Omega\pau{/}\mathrm{Im}(f))+\dim(\Omega\pau{/}\mathrm{Im}(g))\;.\]
In particular, for any $l\leq n$, choosing $m,k\in\mathbb{N}$ such that $m+k=l$, we get that $Q_m\cdot Q_k\subseteq Q_l$. 
%By the same argument, for each $l\in\mathbb{N}$ the preimage of $E_l$ under composition is also open. 
This yields continuity as desired. 
\end{proof}
\textbf{Claim 2:} For all $l<n$ the topologies $\tau_l$ and $\tau_n$ are distinct.
\begin{proof}[Proof of Claim 2] \mic{We use Neumann's Lemma, \cite[Corollary 4.2.2]{HodgesLong}: given finite $A, B\subseteq\Omega$ such that $A\cap\acl(\emptyset)=\emptyset$, there is $g\in G$ such that $g(A)\cap B=\emptyset$.}
\mic{Combining this with the fact that $\acl$ is locally finite and forms a pregeometry,
%~\cite[Corollary 4.2.2]{HodgesLong}), 
it is easy to see by a back \pau{\&} forth argument that the sets $Q_l$ are non-empty  for all $l\in\mathbb{N}$.} Let $f\in Q_n$ and let $\pau{V}$ be an open neighbourhood of $f$ in $\tau_n$; it suffices to show that $\pau{V}\neq Q_n$.  Since $f$ does not belong to any $Q_m$ for any $m<n$, we must have that $\pau{V}\in \tau_0$. Hence, there exist a basic open neighbourhood ${{U}_{(\overline{a}, \overline{b})}\in\pw}$, \pau{where}
\[\pau{U_{(\overline{a}, \overline{b})}:=\{f\in\overline{G}\vert f(\overline{a})=\overline{b}\}\;,}\]
and some set 
\[{W}_{\{c_1,\ldots,c_k\}}=\{f\in\overline{G} \ \vert \ \forall i\; c_i\not\in\mathrm{Im}(f)\}\]
such that $f\in{{U}_{(\overline{a}, \overline{b})}}\cap {W}_{\{c_1,\ldots,c_k\}}\subseteq \pau{V}$. By a back \& forth argument, using Neumann's Lemma and the fact that algebraic closure \mic{is locally finite and} forms a pregeometry, it is easy to construct some $g\in\overline{G}$ such that $g\in \pau{V}$ but $g\not\in Q_n$ by building $g$ so that it agrees with $f$ on $\overline{a}$ and so that 
%its image 
\pau{$\dim(\Omega/\mathrm{Im}(g))>n$ and $\Omega\setminus\mathrm{Im}(g)\supseteq \{c_1,\ldots,c_k\}$.}
%avoids an algebraically closed subset of $\Omega$ of dimension $>n$ containing $\{c_1,\ldots,c_k\}$. 
This implies that $Q_n\neq \pau{V}$, as we wanted to show.% and hence $F_n$ is not open in ${S}_n$. 
\end{proof}
\textbf{Claim 3:} For each $n\in\mathbb{N}$, the  topology $\tau_n$ is Polish. 
\begin{proof}[Proof of Claim 3]
\mic{From~\cite[Lemma 13.2]{Kechris}, we know that if $(X, \tau)$ is a Polish space and $F\subseteq X$ is $\tau$-closed, then the topology generated by $\tau\cup\{F\}$ is also Polish. Hence, 
%By~\cite[Lemma 13.2]{Kechris} 
it is sufficient to prove that $Q_n$ is closed in $\tau_n$  for each $n$.} Note that
%\begin{multline*}
   \[\overline{G}\setminus Q_n:=\{f\in\overline{G}\vert \dim(\Omega\pau{/}\mathrm{Im}(f))>n\} 
   \cup\{f\in\overline{G}\vert \dim(\Omega\pau{/}\mathrm{Im}(f))<n\}\;.\]
%\end{multline*}
The first of the two sets in this union is open in $\tau_0$ and the second is just $\bigcup_{i<n} Q_i$. Hence this set is open in $\tau_n$, as desired. \end{proof}
\end{proof}
It is natural to wonder whether the technique of Lemma~\ref{lem:countablymany} can be extended to all monoids in $\ceemb$. We point out that there are $\omega$-categorical structures for which every non-surjective elementary embedding has co-infinite image, meaning that one cannot define different topologies by open sets of the form of the sets $\pau{Q_n}$. In particular, using a classical construction {of Hrushovski~\cite{omps},}
%~\cite{omps, Wagner:ST}, 
one can build an $\omega$-categorical graph such that every vertex has infinite valency and such that for any two vertices {and any path of length two connecting them, the vertex in the middle of the path belongs to their algebraic closure} %at distance two from each other, the vertex in the path of length two between them belongs to their algebraic closure 
(cf.~\cite[Theorem 5.1]{marimon2025invariant}). From this, 
%and the fact that elementary embeddings preserve algebraic closures, 
 it is easy to prove that any of its non-surjective elementary embedding has co-infinite image. {Thus, we ask the following:}
 
 % Hence, it is natural to ask:
\begin{question} Does every $\overline{G}\in\ceemb$ have infinitely many  Polish semigroup topologies?
\end{question}

\subsection{Minimality of \texorpdfstring{$\pw$}{the topology of pointwise convergence} and the Zariski topology}
%\subsection{The Zariski topology}

\begin{definition} Let ${S}$ be a monoid. The \textbf{semigroup Zariski} topology  $\tau_Z$, introduced in~\cite{EJMMMP-zariski},  is the topology on ${S}$ for which a subbasis is given by all  sets 
\[M_{\phi, \psi}:=\{s\in{S} \ \vert \phi(s)\neq\psi(s)\}\;,\]
where $\phi, \psi: {S}\to{S}$ are functions of the form 
\[s\stackrel{\phi}{\mapsto} t_k s t_{k-1} s\dots t_1 s t_0 \text{ and } s\stackrel{\psi}{\mapsto} q_l s q_{l-1} s\dots q_1 s q_0\;,\]
for fixed $k,l\geq 1$ and $t_0, \dots, t_k, q_0, \dots q_l\in{S}$.
\end{definition}

{When ${S}$ is 
%even 
a group, %then 
a potentially finer topology can be defined by allowing inverses of $s$ to appear in the definitions of $\phi$ and $\psi$; this topology is the 
%classical 
\textbf{group Zariski topology}, and we shall not consider it here.}

Contrary to the pointwise convergence topology $\pw$ on a transformation monoid $\cS$, the semigroup Zariski topology $\zar$ is defined in purely algebraic terms rather than {from} the action of ${S}$. In particular, any isomorphism between monoids is a homeomorphism with respect to this topology. In~\cite{EJMMMP-zariski}, it was pointed out that although $\tau_Z$ need not be Hausdorff or a semigroup topology, it is contained in any topology enjoying these properties; in particular,  $\zar\subseteq \pw$. In fact, in~\cite{EJMMMP-zariski,EJMPP} it was shown that $\zar= \pw$ for the endomorphism monoids of several $\omega$-categorical structures, and a general condition for this to happen was provided  in~\cite{PINSKER_SCHINDLER_2023}. An $\omega$-categorical structure  $\sA$ has a \textbf{mobile core} if any of its elements is contained in the image of an endomorphism into a  substructure which is a model-complete core. Trivially, this is satisfied if $\sA$ is itself a model-complete core.
%\begin{definition}
    %Let $\sA$ be an $\omega$-categorical structure. A \textbf{model-complete core  of $\sA$} is any substructure $\sC$ of $\sA$ such that there is an endomorphism of $\sA$ into $\sC$ and with the property that all endomorphisms of $\sC$ are locally automorphisms (of $\sC$ or $\sA$; the two conditions turn out  equivalent~\cite{BKOPP-equations}). Every $\sA$ has a model-complete core, and it is $\omega$-categorical and unique up to isomorphism~\cite{Cores-journal}. The structure  $\sA$ is said to have a \textbf{mobile core} if any of its elements is contained in the image of an endomorphism into a  model-complete core.
%\end{definition}
%A structure is called a model-complete core if it is isomorphic to its model-complete core; this is the case if and only if its endomorphism monoid has a dense set of invertibles. Trivially, such structures have a mobile core. {Note that every $G\in\caut$ is the automorphism group of a model-complete core by considering the canonical structure for $G$.}

\begin{theorem}[\cite{PINSKER_SCHINDLER_2023}]\label{thm:zariskipinskschind}
  Let $\sA$ be an $\omega$-categorical structure without algebraicity which has a mobile core. Then $\zar=\pw$ on $\End(\sA)$ if one of the following
two conditions holds:
\begin{itemize}
    \item  the model-complete core of $\sA$ is finite;
    \item the model-complete core of $\sA$ is infinite and does not have algebraicity.
\end{itemize}
\end{theorem}

In~\cite{PINSKER_SCHINDLER_2023}, also the first example of an $\omega$-categorical structure $\sA$ for which $\zar\neq\pw$ on $\End(\sA)$ was provided, answering a question from~\cite{EJMPP}. We remark that another albeit less natural  example could actually be obtained from Theorem~\ref{thm:nonreconstruction}
\cite{BodirskyEvansKompatscherPinsker}: clearly, if $\End(\sA)$ and $\End(\sB)$ are isomorphic but not topologically isomorphic, $\pw\neq\zar$ must hold for one of the two monoids. This example also answers positively~\cite[Question 4.1]{PINSKER_SCHINDLER_2023} asking about the existence of an $\omega$-categorical structure with a Hausdorff semigroup topology on its endomorphism monoid which is not finer than $\pw$. Indeed, using a non-topological isomorphism between $\End(\sA)$ and $\End(\sB)$ to send $\pw$ from one monoid onto the other, one sees that one of the two structures must be a positive example (although we do not know which one). However, the following strengthening of the conclusion of the question remains open: %{This example also answers positively~\cite[Question 4.1]{PINSKER_SCHINDLER_2023}, but the following slight variant of the question remains open:}

\begin{question}
{Is there $\cS\in\cend$ such that there is a Hausdorff semigroup topology on $\cS$ which is strictly coarser than $\pw$?}
    %Is there an $\omega$-categorical structure $\sA$ such that there is a Hausdorff semigroup topology on $\End(\sA)$ which is strictly coarser than $\pw$? 
\end{question}

\mic{Recently, the original example from~\cite{PINSKER_SCHINDLER_2023} has been vastly generalised:} %yielding a large number of further instances of $\zar\neq\pw$:}

%vastly generalizing 
%the example from~\cite{PINSKER_SCHINDLER_2023}, a large number of further instances of $\zar\neq\pw$ has been announced recently:
\begin{theorem}[\cite{new}]\label{thm:newres}
{Let $G\in\cgroup$ have locally finite algebraic closure and non-trivial center. Then $\zar$ is not Hausdorff on $\overline{G}$, and hence in particular $\zar\neq \pw$.}
%Let $\sA$ be a model-complete core with locally finite algebraic closure and whose automorphism group has a non-trivial center. Then $\zar$ is not Hausdorff on $\End(\sA)$, and hence in particular $\zar\neq \pw$.
\end{theorem}
\mic{Nevertheless,~\cite{new}  show by different techniques that $\pw$ has no strictly coarser Hausdorff semigroup topology in several monoids in $\ceemb$ whose underlying group has non-trivial centre, such as the monoids of elementary embeddings of countable vector spaces over finite fields or of infinitely many disjoint copies of $\mathbb{K}_2$. Indeed, they prove minimality of $
\pw$ on $\EEmb(\sA)$ for saturated $\sA$ under relatively simple  assumptions on the closure operator $\acl$.}

 \mic{Let us conclude with a remark on minimality of $\pw$ for clones. A function clone  is \textbf{transitive} if the permutation group of its invertible
unary functions acts transitively on its domain.} \mic{For clones satisfying the following weakening of transitivity, minimality of $\pw$ can be lifted from the monoid of its unary functions.} 
Call a semigroup {$\cS$ acting on $\Omega$}
%$\cS\in\ctcl$  
\textbf{weakly directed} if for all $a,b\in\Omega$ there are $f,g\in\cS$ and $c\in\Omega$ such that $f(c)=a$ and $g(c)=b$. %Any transformation monoid containing a transitive permutation group is weakly directed. 
 A clone is weakly directed if the monoid of its unary functions is.
\begin{lemma}[{\cite[Lemma 6.3]{EJMPP}}]\label{lem:liftmini}
Let $\cC\in\cclone$ be weakly directed. Suppose that $\pw$ is minimal amongst Hausdorff semigroup topologies on the monoid of unary operations of $\cC$. Then, $\pw$ is minimal amongst Hausdorff clone topologies on $\cC$.
\end{lemma}

%\begin{itemize}
%    \item Automatic continuity, unique Polish topologies, and Zariski topologies on
%monoids and clones \cite{EJMMMP-zariski}
%    \item Polish topologies on endomorphism monoids of relational structures~\cite{EJMPP}
%    \item On the Zariski topology on endomorphism monoids of omega-categorical structures~\cite{PINSKER_SCHINDLER_2023}.
 %   \item Topologies on endomorphism monoids: find a
%better title!
%\end{itemize}

\subsection{Maximality of \texorpdfstring{$\pw$}{the topology of pointwise convergence}, Property X, and automatic continuity} 

\begin{definition}\label{def:propertyX} Let $\cS$ be a topological semigroup and $N$ a subset of $\cS$. We say that $\cS$ has \textbf{Property~$\bf \mathrm{X}$} with respect to $N$ if the following holds:
for every $s\in\cS$ there are $f_s, g_s\in\cS$ and $t_s\in N$ such that $s=f_st_sg_s$ and for every neighbourhood $\cU$ of $t_s$, the set $f_s(\cU\cap N)g_s$ is a neighbourhood of $s$.
\end{definition}

\begin{theorem}[{\cite[Theorem 3.1]{EJMMMP-zariski}}]\label{thm:propertyXmain} Let $\cS$ %$(\cS, \tau)$ 
be a topological semigroup with topology $\tau$ and $N\subseteq \cS$. If $\cS$ has Property~$\mathrm{X}$ with respect to $N$, then:
\begin{enumerate}
    \item if $\tau$ is Polish and $N$ is a Polish subgroup of $\cS$ then $\tau$ is maximal among the Polish semigroup topologies for $\cS$;
    \item if $N$ is a subsemigroup of $\cS$ with automatic continuity with respect to a class $\classC$ of topological semigroups, then $\cS$ has automatic continuity with respect to $\classC$.
\end{enumerate}
    
\end{theorem}
{The following notion was introduced in~\cite{CameronNesetril}:}

\begin{definition}
    A countably infinite relational structure $\sA$ is \textbf{homomorphism-homogeneous} (H-homogeneous) if every homomorphism between %finitely generated 
    finite substructures of $\sA$ extends to an endomorphism of the whole structure.
\end{definition}

\begin{definition}
  A Fra\"{i}ss\'{e} class $\cK$ in a  relational language has the \textbf{strong amalgamation property with homomorphism gluing} (SAPHG) if for any $\sA, \sB_1, \sB_2\in\cK$ and embeddings $f_i:\sA\to \pau{\sB_i}$ for $i\in\{1,2\}$, there is a {strong amalgam}, i.e.~$\pau{\sC}\in\cK$ and embeddings $g_i:\pau{\sB}_i\to \pau{\sC}$ for $i\in\{1,2\}$ such that $g_1\circ f_1=g_2\circ f_2$ and $g_1(\sB_1)\cap g_2(\sB_2)=g_1f_1(\sA)$, with the property that whenever $h$ is a map from a subset of $\sC$ to some $\mathbb{D}\in\cK$ containing $g_1\circ f_1(\sA)$ in its domain and $h \circ g_1$ and $h \circ g_2$ are homomorphisms, then $h$ is also a homomorphism.
\end{definition}
{We remark that the SAPHG implies that the amalgam $\mic{\mathbb{C}}$ in the definition is ``canonical'' in a sense that agrees with several notions of canonical amalgamation in the literature~\cite{tent2013isometry,PaoliniShelahSIP, kaplan2019automorphism}; the SAPHG is strictly stronger than the existence of a canonical amalgam.}
%\begin{remark} The definition of the SAPHG implies that the amalgam $C$ in the definition is ``canonical'' in a sense that agrees with several notions of canonical amalgamation in the literature~\cite{tent2013isometry,PaoliniShelahSIP, kaplan2019automorphism}. Note the age of $(\mathbb{N}, \neq)$ does not have SAPHG: given any strong amalgam $C$ of $\{1\}$ and $\{2\}$ over $\emptyset$,  the constant map $h:C\to \{1\}$ is not a homomorphism, but its restrictions to $\{1\}$ and $\{2\}$ are both homomorphisms.
%\end{remark}
%\rpau{Happy to shorten/remove the above remark for space}

\begin{theorem}[\cite{EJMPP}]\label{thm:SAPHG} Let $\sA$ be a homogeneous and H-homogeneous structure whose age has SAPHG. Then, $\End(\sA)$ equipped with $\pw$ has Property~$\mathrm{X}$ with respect to $\Aut(\sA)$.     
\end{theorem}

\begin{example}\label{ex:upp} The assumptions of Theorem~\ref{thm:SAPHG} apply to the following structures:
\begin{enumerate}
    \item \pau{the random graph;} %$\pau{\mathbb{R}_{\mathrm{ado}}}$;
    \item the generic directed graph;
    \item the disjoint union $\omega \pau{\mathbb{K}_n}$ of countably many copies of the complete graph $\pau{\mathbb{K}_n}$;
        \item the generic strict poset;
\item any of the structures above including loops;
    \item the generic equivalence relation with $n$-many countably infinite classes $\pau{\mathbb{E}_{n, \omega}}$;
    \item the Cherlin-Hrushovski structure $\sA_{\mathrm{CH}}$.
    %in the language consisting of a $2n$-ary equivalence relation $E_n$ for each $n\geq 1$.
\end{enumerate}
All structures above except the generic strict (or non-strict) poset and $\sA_{\mathrm{CH}}$ are such that their automorphism groups are known to have ample generics, and so automatic continuity with respect to $\csecondgroup$~\cite{HodgesHodkinsonLascarShelah}. With the exception of $\sA_{\mathrm{CH}}$ these examples are already mentioned in~\cite{EJMPP}; the claimed properties are easy to prove for  $\sA_{\mathrm{CH}}$. 
\end{example}
\cite{EJMMMP-zariski} also show that the endomorphism monoid of the countable atomless Boolean algebra $\pau{\sB_\infty}$ has Property~$\mathrm{X}$ with respect to its automorphism group, and the latter is known to have automatic continuity with respect to $\csecondgroup$.

The following proposition can be applied to deduce  that the examples above whose automorphism \pau{groups have} automatic continuity with respect to  $\csecondgroup$ are such that their endomorphism monoids have automatic continuity with respect to $\csecondsemi$:

\begin{proposition}[{\cite[Proposition 4.1]{EJMMMP-zariski}}] Let $G\in\csecondgroup$. Then, $G$ has automatic continuity with respect to $\csecondgroup$ if and only if it has automatic continuity with respect to $\csecondsemi$.
\end{proposition}

{While Property~$\mathrm{X}$ is a powerful tool for showing  maximality of $\pw$ and lifting of automatic continuity from subgroups of a monoid,~\cite{pinsker2023semigroup} note that there is some leeway in Definition~\ref{def:propertyX}: the decomposition of $s\in S$ could be of arbitrary length without changing the truth of Theorem~\ref{thm:propertyXmain}; they call this generalization Property~$\tcl{\mathrm{X}}$. Similarly, in the same definition $s$ can be composed from the left with a left-invertible element of $\cS$, leading to \textbf{pseudo-Property~$\bf \tcl{X}$}. They apply this notion to  $\End(\pau{\mathbb Q;\leq})$ equipped with $\tau_{\mathrm{rich}}$, a finer topology  than $\pw$, and the subgroup $\Aut(\pau{\mathbb Q;\leq})$ equipped with $\pw$ to conclude that any  Polish semigroup topology on $\End(\pau{\mathbb Q;\leq})$ must be contained in $\tau_{\mathrm{rich}}$. By a series of further arguments, of which some can be recast as statements about pseudo-Property~$\tcl{\mathrm{X}}$, and others involve different techniques such as~Baire category, they obtain:
}
\begin{theorem}[\cite{pinsker2023semigroup}]\label{thm:Qupp}
   {$\End(\pau{\mathbb Q;\leq})$ has the UPP.}
\end{theorem}
{\cite{pinsker2023semigroup} note that most of the results on the UPP in Example~\ref{ex:upp} do not use the full strength of Polishness: they have  unique second countable metrisable semigroup topologies. This occurs whenever the UPP can be shown \mic{by} applying the second item of Theorem~\ref{thm:propertyXmain}. The considerably more complex proof of Theorem~\ref{thm:Qupp}, which only partially  fits into this framework, additionally requires the space to be Baire.}

%\rpau{Here we should add a bit about pseudo property $\overline{X}$ and~\cite{pinsker2023semigroup}}
{For the endomorphism monoids of some well-studied structures the  status of the UPP is still open; among homogeneous graphs, the only one for which whether the UPP holds is not known is the following (see~\cite{schindler2023semigroup}):}

\begin{question} Does the endomorphism monoid of the complete multipartite graph $\mathbb{K}_{\omega, \omega}$ (with $\omega$-many blocks of size $\omega$) in the language of graphs have UPP?
\end{question}

{Property~X can also be applied in the context of clones. Given a topological clone $\cC$ with a topology $\tau$, we can associate with it a topological semigroup $\cS_{\cC}$ with domain $\cC$, topology $\tau$, and the operation $\ast$ given by $f\ast g=f(g, \dots, g)$. Whilst the algebraic structure of $\cS_{\cC}$ is weaker than that of $\cC$ in the sense that not every semigroup homomorphism with respect to $\ast$ is a clone homomorphism, it is often sufficiently rich to restrict the possible clone topologies on $\cC$:}

\begin{lemma}[{\cite[Lemma 7.1]{EJMMMP-zariski}}] Let $\cC$ be a topological clone and let $\cS_{\cC}$ have automatic continuity with respect to $\csecondsemi$. Then, $\cC$ has automatic continuity with respect to $\csecondclone$.
\end{lemma}
In~\cite{EJMMMP-zariski}, the above lemma and Property~X are used to show that $\pw$ is the unique second countable Hausdorff clone topology on the polymorphism clone of the countable atomless Boolean algebra. In~\cite[Theorem 6.2]{EJMPP}, it is shown that homogeneous, H-homogeneous structures $\sA$ with the SAPHG and whose age is closed under finite non-empty direct products are such that $\cS_{\Pol(\sA)}$ %(with $\pw$) 
has Property~X with respect to $\Aut(\sA)$. This is then applied to prove UPP and automatic continuity with respect to $\csecondclone$ for various structures, such as the random graph and $\omega \mic{\mathbb{K}}_\omega$ with loops. We invite the reader to compare this result also to~\cite[Theorem 4.1]{PechPechHomeo}. In the latter theorem, the authors introduce the \textbf{amalgamated extension property} (AEP)~\cite[Definition 3.23]{PechPechHomeo}, a weaker amalgamation property than the SAPHG. They show that homogeneous, H-homogeneous structures $\sA$ with the AEP whose age is closed under finite non-empty direct products, and whose endomorphisms contain all constant functions, are such that $\Pol(\sA)$ has automatic homeomorphicity with respect to $\cclone$. The latter result is based on ``gate covering'' techniques of~\cite{Reconstruction}, but, unlike the original paper, does not rely on upwards transfers of automatic homeomorphicity from the underlying automorphism group.

\subsection{More on automatic continuity and open mappings}

Preceding the positive instances in Example~\ref{ex:upp}, automatic continuity for semigroups was believed to be a less fruitful notion than for groups due to the following result:

\begin{theorem}[{\cite[Corollary 10]{Reconstruction}}] No $\overline{G}\in\ceemb$ has automatic continuity with respect to $\csemi$.
%Let $\sA$ be an $\omega$-categorical structure. Then $\Emb(\sA)$ and $\EEmb(\sA)$ do not have automatic continuity with respect to $\csemi$.
\end{theorem}
%\rpau{Maybe this could just be phrased in terms of injective monoids.}

{The small index property for groups can be meaningfully extended to semigroups. A \textbf{right congruence} on a semigroup $\cS$ is an equivalence relation $\mathcal{E}\subseteq \cS\times\cS$ such that whenever $(g,f)\in\mathcal{E}$ and $h\in{S}$, then $(hg, hf)\in\mathcal{E}$. We define a left-congruence analogously but with respect to composition from the other side.} 
%Below, we keep naming conventions coming from semigroup theory, though we write composition in the opposite way than semigroup theorists. In particular, a \textbf{right congruence} on a semigroup $\cS$ is an equivalence relation $\mathcal{E}\subseteq \cS\times\cS$ such that whenever $(g,f)\in\mathcal{E}$ and $h\in{S}$, then $(hg, hf)\in\mathcal{E}$. We define a left-congruence analogously but with respect to composition (for us) on the right. 

\begin{definition}A  topological semigroup $\cS$ has the \textbf{right small index property} (rSIP) if every right congruence with countably many classes is open in $\cS\times\cS$. We define analogously the \textbf{left small index property} (lSIP). 
\end{definition}

\begin{proposition}[{\cite[Proposition 4.3]{EJMMMP-zariski}}] A topological monoid $\cS$ has the rSIP if and only if it has automatic continuity with respect to $\csemi$.
\end{proposition}

\begin{proposition}[{\cite[Corollary 4.5]{EJMMMP-zariski}}] If a topological monoid $\cS$ has automatic continuity with respect to $\csecondsemi$, then it has both the rSIP and the lSIP.
\end{proposition}

Automatic continuity with respect to $\csemi$ and with respect to $\csecondsemi$ are distinct notions for Polish semigroups, and even for closed oligomorphic transformation monoids ~\cite[Proposition 5.18]{EJMMMP-zariski}.
\mic{Moreover, automatic continuity  does not imply automatic homeomorphicity with respect to $\csemi$ for Polish semigroups,  {see~\cite[Theorem 5.15 (vii)]{EJMMMP-zariski}.} The latter  contrasts the situation for Polish groups: by the Open Mapping Theorem~\cite[Theorem 1.2.6]{BeckerKechris}, any continuous surjective homomorphism between Polish groups is open. %, and hence automatic continuity implies automatic homeomorphicity in any subclass of Polish groups. 
 It is not hard to see that the class of closed transformation monoids does not satisfy an analogous Open Mapping Theorem, see e.g.~\cite{grilj}; however, we do not know the following.}
\begin{question}\label{q:acandah} Is there $\cS\in\csemi$ with automatic continuity with respect to $\csemi$ but without automatic homeomorphicity with respect to $\csemi$? {If so, is there an oligomorphic example?} %(Analogously, one could ask the same question with respect to $\cend$). 
\end{question}

{Regarding the openness of clone homomorphisms, we have both a lifting result from semigroups as well as a result specific to clones.
%; we start with the former. Namely, 
%{Firstly,} 
The former implies that for isomorphisms from a transitive clone $\cC$, it is sufficient to check openness on its unary functions.}
%For transitive clones, one can lift openness from openness of a clone isomorphism from openness of the induced isomorphism between monoids of unary functions:

\begin{theorem}[\cite{Reconstruction}] Let $\cC, \cD\in\cclone$. Let $\xi:\cC\to\cD$ be a clone homomorphism and let $\xi'$ be its restriction to unary functions.
\begin{itemize}
    \item Suppose $\cD$ is weakly directed. If $\xi'$ is continuous, then so is $\xi$;
    \item Suppose $\cC$ is weakly directed and $\xi$ is an isomorphism. If $\xi'$ is open, so is $\xi$.
\end{itemize}
\end{theorem}

{The following trick for clones is not available for monoids, since homomorphisms of the latter do not necessarily preserve constant functions:}
%{To prove openness of a clone homomorphism there are also some tricks which are not available for monoids, since the latter do not necessarily distinguish constant functions with their algebraic structure:}

\begin{proposition}[{\cite{Reconstruction}}]\label{prop:constants} Let $\cC\in\cclone$ contain all constant functions. Then, any isomorphism $\xi:\cC\to\cD\in\cclone$ is open.
\end{proposition}

\section{Rubin's dream: stronger notions of reconstruction}

The group $\somega$ acts continuously on $\Omega^\Omega$ by conjugation: $(\theta,f)\mapsto f^\theta:=\theta f \theta^{-1}$ for  $\theta\in\somega$ and  $f\in\omom$. Hence, for any closed subsemigroup ${S}$ of $\omom$ and any $\theta\in\somega$, we obtain a new closed subsemigroup ${S}^\theta:=\{f^\theta\;|\; f\in{S}\}$ which is topologically isomorphic to it by fixing $\theta$ in the above map. We call such homeomorphisms \textbf{set-induced}. This idea leads to a stronger notion of reconstruction:

\begin{definition} Let $\cS\in\ctcl$ and $\classC\subseteq\ctcl $. We say that $\cS$ has \textbf{automatic action compatibility} with respect to $\classC$ if  any algebraic isomorphism from $\cS$ onto an element of $\classC$ is set-induced.
\end{definition}
\noindent Again, the analogous notion can be defined also in the context of permutation groups and clones. In the latter case, the action on $\Omega^{\Omega^n}$ is \[(\theta,f)\mapsto f^\theta(x_1,\ldots,x_n):=\theta f (\theta^{-1}(x_1),\ldots,\theta^{-1}(x_n))\;.\]
These notions have a structural interpretation: we call structures $\sA,\sB$ on $\Omega$ \textbf{(fo/e/ep/pp)-bi-definable} if they are (fo/e/ep/pp)-bi-interpretable with $d=1$ and a bijection $I\colon \Omega\to\Omega$ (using the notation from~Definition~\ref{def:interpretation}). In the situation where they are ep-bi-definable, for $\theta:=I$ we have $\End(\sB)=\End(\sA)^\theta$. Analogous statements hold for the other notions of definability. Moreover, for all of these notions the converse (i.e.~the existence of a permutation $\theta$ moving one space of symmetries onto the other by conjugation) holds in the oligomorphic case. This follows from the theorem of Ryll-Nardzewski, Engeler, and Svenonius for automorphism groups, is due to~\cite{BodirskyNesetrilJLC} for pp-bi-definability, and is  folklore for the other cases.
%\rpau{This paragraph is now quite unclear.}

For the above form of reconstruction to hold, we need to exclude algebraicity: for any oligomorphic permutation group $G$ we can construct a topologically isomorphic oligomorphic $G'$ whose action is not isomorphic to that of $G$ by adding a fixed point. Thus, the interesting classes are $\cnagroup,\cnaeemb, \cnasemi,$ and $\cnaclone$: spaces of symmetries of $\omega$-categorical structures without algebraicity. %, respectively, of closed oligomorphic groups, closures of oligomorphic groups in $\Omega^\Omega$, closed oligomorphic transformation monoids, and closed oligomorphic clones whose underlying group of invertible elements has no algebraicity. 
In this context, we already have the following consequence of Theorem~\ref{thm:zariskipinskschind} regarding automatic homeomorphicity:

%Note that if $\Aut(\sA)$ has AAR with respect to $\cnagroup$, or if $\EEmb(\sA)$ has AAR with respect to $\cnaeemb$ it means that from the algebraic structure of $\Aut(\sA)$ (or $\EEmb(\sA)$), we can reconstruct $\sA$ up to interdefinability amongst $\omega$-categorical structures with no algebraicity. Similarly, for $\End(\sA)$ to

%The following is a direct consequence of Theorem~\ref{thm:zariskipinskschind}:
\begin{corollary}\label{cor:ahfromzariski} Let $\overline{G}\in\cnaeemb$. Then, $\overline{G}$ has automatic homeomorphicity with respect to $\cnaeemb$.
%Let $\sA$ and $\sB$ be $\omega$-categorical structures with trivial algebraicity. Suppose there is an algebraic isomorphism $\eta:\EEmb(\sA)\to\EEmb(\sB)$. Then $\eta$ is a homeomorphism.\remmic{I guess we should formulate this differently: any element of the class of closed oligomorphic monoids with no algebraicity (cannot guess the macro) has a.h. with respect to that class.}
\end{corollary}
\begin{proof}
%\remmic{and then the by now already known canonical language comes into play...}\rpau{We actually never mention the canonical language before this point now.}
By Theorem~\ref{thm:zariskipinskschind},$\zar=\pw$ for any structure in $\cnaeemb$, yielding that any isomorphism between two structures in the class is also topological.
%Note that we may consider $\sA$ and $\sB$ in languages in which they are model complete cores so that $\End(\sA)=\EEmb(\sA)$ and $\End(\sB)=\EEmb(\sB)$. Now, by Theorem~\ref{thm:zariskipinskschind} since $\zar=\pw$ for both monoids, $\eta$ yields a bijection between the open sets of $\EEmb(\sA)$ and $\EEmb(\sB)$.
\end{proof}

\mic{However, much more is true. Namely, it is a result of~\cite{Rubin} that any fo-bi-interpretation of $\sB$ in $\sA$, where $\sA,\sB$ are $\omega$-categorical structures without algebraicity, 
%is 
{gives rise to} a bi-definition: this is essentially Lemma~2.11 in~\cite{Rubin}. In some of the literature on reconstruction one gets the impression that this fact has been overlooked, perhaps due to the complex mathematical language used in Rubin's paper and the belief that the statement holds only in the presence of a weak $\forall\exists$-interpretation (the main topic of  Rubin's article). Be that as it may, the following holds by analysing Rubin's proof, see~\cite{RomanMichael}:}
%Be that as it may, the same statement holds for  all stronger (i.e.~e/ep/pp) notions of bi-interpretation  by analysing Rubin's proof, and a new proof has been announced in~\cite{RomanMichael}{:}
%.

\begin{theorem}[{\cite{Rubin}}]\label{thm:bidefinition} 
\mic{Let $G, H\in \cnagroup$.
%\pau{have no constant functions}. 
Then any topological isomorphism from $G$ onto $H$ is set-induced. The same statement holds in $\cnaeemb$.}
%\pau{(without the condition on constant functions)}.}
%Let $\cS, \cT\in \cnasemi$.
%\pau{have no constant functions}. 
%Then any topological isomorphism from ${S}$ onto $\mathcal T$ is set-induced. Similar statements hold in $\cnagroup$ and $\cnaclone$.
%\pau{(without the condition on constant functions)}.
\end{theorem}

 %Theorem~\ref{thm:bidefinition} has, in the case of groups, been exploited to show that the equivalence relation of topological isomorphism on the class $\cnagroup$ of closed oligomorphic groups without algebraicity is \textbf{smooth}~\cite{paolini2024two}. 
\mic{From the proof given in~\cite{RomanMichael},  Theorem~\ref{thm:bidefinition} straightforwardly extends to elements of $\cnasemi$ whose finite-to-one functions are dense, and to transitive {model-complete core clones in $\cnaclone$.}
%whose unary functions are in $\cnaeemb$.
However, whether Theorem~\ref{thm:bidefinition} holds for  $\cnasemi$ or $\cnaclone$ with no further assumptions remains an intriguing open problem:}
\begin{question} Are topological isomorphisms between members of $\cnasemi$ always set-induced? What about $\cnaclone$?
\end{question}

\mic{Theorem~\ref{thm:bidefinition} has, in the case of groups, been exploited to show that the equivalence relation of topological isomorphism on the class $\cnagroup$ of closed oligomorphic groups without algebraicity is \textbf{smooth}~\cite{paolini2024two}. Similar ideas are used in~\cite{RomanMichael} to prove smoothness of topological isomorphism on $\cnaeemb$ and for the class of transitive clones in $\cnaclone$ whose unary part is in $\cnaeemb$.}

%By~\cite{RomanMichael}, the same is true for monoids (i.e.~the class $\cnasemi$) and clones (i.e.~the class $\cnaclone$). Both works contain generalisations beyond the context of no algebraicity. 
%As a consequence of Theorem~\ref{thm:bidefinition} and Corollary~\ref{cor:ahfromzariski}, automatic homeomorphicity and automatic action compatibility coincide in the context of $\cnaeemb$:
{As a consequence of Theorem~\ref{thm:bidefinition}, automatic homeomorphicity and automatic action compatibility coincide in the context of $\cnaeemb$. Together with Corollary~\ref{cor:ahfromzariski} this yields:}

\begin{corollary}\label{cor:aacclosures} Let $\overline{G}\in\cnaeemb$. Then, $\overline{G}$ has automatic action compatibility with respect to $\cnaeemb$.
\end{corollary}

Some specific cases of Corollary~\ref{cor:aacclosures} were proven in~\cite{behrisch-vargasgarcia-stronger}. The following proposition allows to lift being set-induced from semigroup homomorphisms to clone homomorphisms.

\begin{proposition}[{\cite[Theorem 4.3]{behrisch-vargasgarcia-stronger}}] Let $\cC, \cD\in\cclone$ with $\cC$ being weakly directed. Let $\xi:\cC\to\cD$ be a surjective clone homomorphism and let $\xi'$ be its restriction to unary functions. If $\xi'$ is set-induced, then so is $\xi$.
\end{proposition}

%\begin{theorem}[{\cite[Theorem 4.3]{behrisch-vargasgarcia-stronger}}] Let $\cC,\cD\in\cclone$ be clones, and suppose that $\cC$  is weakly directed. Let  $\xi:\cC\to\cD$ be a surjective clone homomorphism, and suppose its restriction to unary functions is set-induced. Then $\xi$ is set-induced.%, for all $h\in\cC\cap\mathcal{O}^{(n)}$ and $n\in\mathbb{N}$, $\xi:h\mapsto \theta \cdot h\cdot (\theta^{-1}\times \theta^{-1})$.
%\remmic{Does this actually work for functions of arity $>2$?}\rpau{Pretty sure the answer is yes.}
%\end{theorem}

We can thus extend Corollary~\ref{cor:aacclosures} to transitive clones.
\begin{corollary} 
Let $\cC,\cD\in\cnaclone$ be   model-complete core clones. If $\cC$ is transitive, then any clone isomorphism $\xi:\cC\to\cD$ is set-induced.

%Let $\cC\in\cnaclone$ be such that $\cC^{(1)}$ is transitive and in $\cnaeemb$. Let $\cD\in\cnaclone$ be such that $\cC^{(1)}\in\cnaeemb$. Then,  any clone homomorphism $\xi:\cC\to\cD$ is set-induced.
\end{corollary}

%\rpau{I think the corollary above is optimal: one can build orbit-semiprojections which give rise to algebraically isomorphic clones which are distinct function clones.}\remmic{I tried quickly, doesn't seem obvious.}

\section{Final remarks on clones}

In this final section, we summarise some of the literature on  reconstruction for polymorphism clones mentioned throughout this survey and list some final questions on the topic. Current techniques for clones usually rely on lifting reconstruction properties from the underlying group or monoid. Whilst this was originally the case also for the study of  reconstruction for monoids, now we have several powerful direct techniques. We hope that the study of clones will develop in a similar direction. 
%than the techniques for monoids we covered. We hope that as the field progresses there will be more techniques that do not rely on a bottom-up approach. 
The original ``gate covering'' techniques of~\cite{Reconstruction} relied on lifting to prove automatic homeomorphicity with respect to $\cclone$ of the Horn clone, the polymorphism clone of the random graph, and all minimal closed clones above the automorphism group of the random graph. Meanwhile, ~\cite{behrisch-truss-vargas, truss2021reconstructing} use Proposition~\ref{prop:constants} and other lifting tricks to deduce automatic continuity of the polymorphism clones of  $\pau{(\mathbb{Q}; \leq)}$ \mic{and some structures first-order definable in it.}
%and the other structures  mentioned in Theorem~\ref{thm:reductsofq}.
Although Property~$\mathrm{X}$ for $\cS_{\Pol(\sA)}$ can be used to deduce UPP and automatic continuity for $\Pol(\sA)$~\cite{EJMMMP-zariski, EJMPP}, current proofs of UPP do rely on lifting minimality of $\pw$ using Lemma~\ref{lem:liftmini}, and automatic continuity still relies on lifting from $\Aut(\sA)$. {The only existing technique that seems to not rely on lifting is~\cite[Theorem 4.1]{PechPechHomeo}, which 
%shows 
utilises gate coverings to show automatic homeomorphicity with respect to $\cclone$ for the polymorphism clones of various homogeneous structures, such as the generic (non-strict) poset.}\\

%Meanwhile,~\cite{Reconstruction} developed the ``gate covering'' technique for proving  automatic homeomorphicity of closed clones with respect to $\cclone$, and applied it to the Horn clone, the polymorphism clone of the random graph, and all minimal closed clones above the automorphism group of the random graph. Surprisingly, as noted in~\cite[Theorem 4.1]{PechPechHomeo}, gate covering techniques can also be used to prove automatic homeomorphicity 
%with respect to $\cclone$ 
%without (at least explicitly) relying on automatic homeomorphicity of the underlying group or monoid. Property~$\mathrm{X}$ for $\cS_{\Pol(\sA)}$ can be used to deduce UPP for $\Pol(\sA)$~\cite{EJMMMP-zariski, EJMPP}. However, current proofs of UPP do rely on lifting upwards minimality of $\pw$ using Lemma~\ref{lem:liftmini}. Finally, conditionally on automatic continuity of $\Aut(\sA)$ with respect to $\csecondsemi$, Property X for $\cS_{\cC}$ yields automatic continuity for $\Pol(\sA)$ with respect to $\csecondclone$.}\\

{One challenge 
%we are  facing 
{we face} is that current techniques often rely on specific structural properties such as particular forms of amalgamation. However,  most structures do not enjoy these properties themselves, although in many cases they are first-order definable, or interpretable, in  another structure that does. For example, amongst the continuum  of closed clones  whose underlying permutation group is $S_\Omega$ (presented in ~\cite{BodChenPinsker}), we only know automatic homeomorphicity for those containing the full monoid $\pau{O}^{(1)}$, for the Horn clone~\cite{Reconstruction}, and for $\mic{\cclosure{S_\Omega}}$~\cite{behrisch-truss-vargas}.%\remmic{Is this really due to BTV?}}\rpau{I think so\dots, I don't see this in~\cite{Reconstruction} and they prove a general statement on this.}

%A challenge to the study of topologies on clones is that current techniques rely on specific structural properties that do not cover the large classes of clones that appear even in apparently tame contexts. For example, amongst function clones whose underlying permutation group is $S_\Omega$, we only know automatic homeomorphicity for those containing the full monoid $\mathcal{O}^{(1)}$, for the Horn clone~\cite{Reconstruction}, and for $\cclosure{S_\Omega}$~\cite{behrisch-truss-vargas}. However, there are continuum many clones whose unary functions correspond to the monoid of injections $\overline{S_{\Omega}}$~\cite{BodChenPinsker}, meaning that the following question is still open:
\begin{question} {Does the polymorphism clone of every structure which is first-order definable in $(\mathbb{N}, =)$ have automatic homeomorphicity with respect to $\cclone$? In general, are there conditions which guarantee that automatic homeomorphicity of $\End(\sA)$ (or $\Pol(\sA)$) for a structure $\sA$ \pau{implies}
%imply
the same for the structures first-order definable or interpretable in it?}
\end{question}
{Note that Theorem~\ref{thm:ACdoesnotimplyAH} shows that some non-trivial conditions must be imposed. We are not aware of any model-complete core clone without automatic homeomorphicity whose monoid of unary functions has automatic homeomorphicity.}\\
%\rpau{The second question has a negative answer: the counterexample we give to AH is a reduct of a structure which does have AH}\remmic{It says `when'... but we should mention the counterex.. We don't have a core counterex, perhaps that's a good question.}

{Of particular interest, especially in applications to Constraint Satisfaction~\cite{Book, pinsker2022current}, is when an $\omega$-categorical structure $\sA$ pp-interprets all finite structures. This is the case if and only if there is a continuous clone homomorphism from $\Pol(\sA)$ to the clone $\cP_{\{0,1\}}$ of projections on a two-element set. Naturally, one then wishes to know whether all such homomorphisms must be continuous, or whether at least the existence of a homomorphism implies the existence of a continuous one. {A structure first-order definable in the Cherlin-Hrushovski structure provides a counterexample to the former statement~\cite{BPP-projective-homomorphisms}. However, the proof requires the ultrafilter lemma, and the structure is only homogenisable in an infinite language. } 
%The Cherlin-Hrushovski structure provides a counterexample to the former statement~\cite{BPP-projective-homomorphisms}, but firstly the proof requires the ultrafilter lemma, and secondly it is only homogeneous in an infinite language. 
The truth of the latter statement is unknown, cf.~\cite[Questions~7.3~\&~7.4]{BPP-projective-homomorphisms}.}
\begin{question} 
{
Let $\cC\in\cpol$ have a clone homomorphism to ${P}_{\{0,1\}}$. Does it then always also have a continuous such homomorphism? Can the existence of a discontinuous homomorphism from some $\cC\in\cpol$ to $\cP_{\{0,1\}}$ be proven in ZF? Can it be proven for the polymorphism clone of a structure which is first-order definable in a finitely homogeneous structure?
}
\end{question}

{A positive answer to the first question is made plausible by the result from~\cite[Theorem 1.6]{Topo} that for an $\omega$-categorical \pau{model-complete} core $\mic{\mathbb{A}}$, 
$\Pol(\sA,c_1,\ldots,c_m)$  has a continuous homomorphism to $\cP_{\{0,1\}}$ for some constants $\pau{c_1, \dots, c_m}$ if and only if $\Pol(\sA, d_1, \dots, d_n)$ has a homomorphism to $\cP_{\{0,1\}}$ for some constants $\pau{d_1, \dots, d_n}$.}\\

%It would be interesting to understand whether the requirement of a uniforml

%In particular, conjecturally, the hardness of an infinite-

%\rpau{Here we should have a final brief overview of results on clones. In particular, I am thinking of~\cite{Reconstruction, PechPechHomeo, EJMMMP-zariski, EJMPP}.}

%\begin{theorem}[\cite{PechPechHomeo}] Let $\sA$ be a homogeneous and H-homogeneous structure whose age is closed under finite products and has AEP. Suppose that $\End(\sA)$ contains all constant functions. Then $\Pol(\sA)$ has automatic homeomorphicity with respect to $\cclone$.
%\end{theorem}
%\input{Sections/Questions}

\textbf{Acknowledgement:}
We thank Luna Elliott, Roman Feller, \mic{and Mira Tartarotti} for helpful comments. 
Funded by the European Union (ERC, POCOCOP, 101071674). Views and opinions expressed are however those of the authors only and do not necessarily reflect those of the European Union or the European Research Council Executive Agency. Neither the European Union nor the granting authority can be held responsible for them. This research was funded in whole or in part by the Austrian Science Fund (FWF) [I 5948]. For the purpose
of Open Access, the authors have applied a CC BY public copyright licence to any Author Accepted Manuscript (AAM) version arising from this submission.

%

%%%%%%%%%%%%%%%%%%%%%%%%%%%%%%%%%%%%%%%%%%%%%%%%%%%%
% PLEASE DO NOT EDIT THE FOLLOWING LINES

 \printbibliography
 %%%%%%%%%%%%%%%%%%%%%%%%%%%%%%%%%%%%%%%%%%%%%%%%%%%%

\end{document}